\documentclass[11pt,leqno]{article}

\usepackage{amssymb,amsmath,amsthm,mysects,url,rotating}
 \usepackage{graphicx,epsfig}
 \usepackage{xcolor,graphicx}
\newcommand{\n}{\noindent}

\newcommand{\vp}{\varepsilon}
\newcommand{\bb}[1]{\mathbb{#1}}
\newcommand{\cl}[1]{\mathcal{#1}}

\theoremstyle{plain}
\newtheorem{thm}{Theorem}[section]
\newtheorem{lem}[thm]{Lemma}

\newtheorem{pro}[thm]{Proposition}

\newtheorem{cor}[thm]{Corollary}

\theoremstyle{definition}

\newtheorem{dfn}[thm]{Definition}

\theoremstyle{remark}
\newtheorem{rem}[thm]{Remark}

\numberwithin{equation}{section}

\def\R{\bb R}
\def\RR{\bb R}
\def\C{\bb C}
\def\CC{\bb C}

\def\E{\bb E}
\def\F{\bb F}
\def\P{\bb P}
\def\T{\bb T}

\def\NN{\bb N}
\def\N{\bb N}
\def\RR{\bb R}
\def\Q{\bb Q}

\def\CC{\bb C}

\def\RR{\bb R}
\def\CC{\bb C}

\def\E{\bb E}
\def\F{\bb F}
\def\P{\bb P}
\def\T{\bb T}

\def\NN{\bb N}
\def\RR{\bb R}

\def\CC{\bb C}

\def\v{\varphi}
\def\ov{\overline}
\begin{document}

\title{Random Matrices   and  Subexponential Operator Spaces}

\author{by\\
Gilles  Pisier\\
Texas A\&M University\\
College Station, TX 77843, U. S. A.\\
and\\
Universit\'e Paris VI\\
IMJ, Equipe d'Analyse Fonctionnelle, Case 186,\\ 75252
Paris Cedex 05, France}

 \maketitle

\begin{abstract} We   introduce and  study  a generalization of the notion of   exact operator space
that we call subexponential.
Using Random Matrices we show that the  factorization results of Grothendieck type that are known in the exact case all extend to the subexponential case, but we exhibit (a continuum of distinct) examples of non-exact subexponential operator spaces, as well as a $C^*$-algebra
that is  subexponential with constant $1$ but not exact.
We also show that $OH$, $R+C$ and $\max(\ell_2)$ (or any other maximal operator space) are not subexponential. \end{abstract}

In \cite{JP,PS} operator space versions of Grothendieck's theorem were proved in the form of a special factorization property for (jointly) completely bounded bilinear forms on $E\times F$
when $A,B$ are $C^*$-algebras and  $E\subset A$, $F\subset B$ are {\it exact} operator subspaces.
In particular, when $E=A,F=B$ this was proved for exact $C^*$-algebras. In \cite{HM2} this last result was extended to arbitrary $C^*$-algebras. A remarkable, considerably simpler proof was recently given in
\cite{RV}. In the case of ``exact" subspaces $E\subset A,F\subset B$, this recent proof from \cite{RV}
deduces the result of \cite{PS} directly from that of \cite{JP}. In this paper we introduce
a larger class of operator spaces, that we call ``subexponential", for which the same
Grothendieck type  factorization property from  \cite{JP,PS} still holds.
The known examples of non-exact operator spaces turn out to be also non-subexponential, but  
 in \S \ref{s6} an example is constructed  showing that the new class is strictly larger than that of exact operator spaces.\\
The definition of ``subexponential" involves the growth of
a sequence of integers $N\mapsto K_E(N,C)$ attached to an operator space $E$ (and a constant $C>1$),
in a way that is similar but seems different from the number $k_E(N,C)$ introduced by us in \cite{Pq2}.
We denote by $K_E(N,C)$
the smallest $K$ such that there is a  linear embedding $f:\ E\to M_K$  satisfying
   $$\forall x\in M_N(E)\quad    \|(Id\otimes f)(x)\|_{M_N(M_K)}   \le  \|x\|_{M_N(E)}\le C  \|(Id\otimes f)(x)\|_{M_N(M_N)} . $$
 The latter sequence
is bounded iff $E$ is $C$-exact  while it is such that $\log K_E(N,C)/ N\to 0$ iff $E$ is $C$-subexponential.
For the non-exact $C^*$-algebra $A$ constructed in \S \ref{s7} we have polynomial growth: we have $K_E(N,1+\vp)\in O(N^d )$ for any finite dimensional $E\subset A$ for some $d$ depending on $E$ and $\vp>0$.
For the non-exact example  in \S \ref{s6}, we   have  $K_E(N,2+\vp)\in O(N^2 )$
and also $K_E(N,2+\vp) \ge c \sqrt{N}$.
There is a notion of ``subexponential constant" analogous to the exactness constant, and we give
estimates from below (of the same order) of that constant for the same examples ($OH_n$, $R_n+C_n$
or maximal spaces) for which lower bounds of
the exactness constant are   known. 

To tackle subexponentiality, we make crucial  use of    Gaussian random matrices
and particularly of Haagerup and Thorbj\o rnsen's  \cite{HT2}.
Let  $Y^{(N)}$ denote a random $N\times N$-matrix with i.i.d. 
  complex Gaussian entries $Y^{(N)}(i,j)$ with $\E Y^{(N)}(i,j)=0$ and $\E| Y^{(N)}(i,j) |^2=1/N$,   and let
     $(Y_j^{(N)})$ be a sequence of i.i.d. copies of $Y^{(N)}$ on some probability space $(\Omega,\P)$.
   Let $E\subset B(H)$ be an operator space. For any $(a_1,\cdots,a_n)\in E^n$
   we define
   \begin{equation}\label{in2}||| (a_1,\cdots,a_n)|||= \limsup\nolimits_{N\to \infty} \left\|\sum_1^n  Y_j^{(N)} \otimes a_j \right\|,\end{equation}
   where-here and below-the norm is the   minimal (or spatial) tensor norm
   (here this is simply the norm of $M_N(B(H)$). 
   Note that   this is non-random. Indeed, 
   by  concentration of measure (see \S \ref{scon})
   we have
   almost surely  \begin{equation}\label{in2con} \limsup\nolimits_{N\to \infty}\left| \left\|\sum_1^n Y_j^{(N)} \otimes a_j\right\|-\E\left\|\sum_1^n Y_j^{(N)} \otimes a_j\right\| \right|=0,\end{equation}
   and hence 
     \begin{equation}\label{inn2}||| (a_1,\cdots,a_n)|||= \limsup\nolimits_{N\to \infty} \E\left\|\sum_1^n Y_j^{(N)} \otimes a_j\right\|.\end{equation}
The main result of Haagerup and Thorbj\o rnsen's \cite{HT3} implies that if $E$ is exact with constant 1
  then
   \begin{equation}\label{in2bis}||| (a_1,\cdots,a_n)|||=\|\sum c_j \otimes a_j\|\end{equation} where $(c_j)$ is a free circular sequence
  in Voiculescu's sense, and the limsup is actually almost surely a limit. 
  Indeed, it can be shown rather easily
   that 
  \begin{equation}\label{es3}\|\sum c_j \otimes a_j\|\le \liminf\nolimits_{N\to \infty} \left\|\sum_1^n Y_j^{(N)} \otimes a_j\right\| \end{equation} 
   holds almost surely for {\it any} $E$. By \eqref{in2con} this boils down to
 \begin{equation}\label{es3c}\|\sum c_j \otimes a_j\|\le \liminf\nolimits_{N\to \infty} \E\left\|\sum_1^n Y_j^{(N)} \otimes a_j\right\|. \end{equation} 
These lower bounds are easy to deduce from the well known weak convergence of 
$(Y_j^{(N)})$ to a free circular system $(c_j)$ due to Voiculescu (see \cite{VDN}), originating in Wigner's famous result for a single matrix $(Y^{(N)})$. 
This weak convergence asserts that 
if $\tau_N$ (resp. $\tau$) denotes the normalized trace on $M_N$
(resp. on the von Neumann algebra generated by $(c_j)$) , we have a.s.
\begin{equation}\label{voi}\lim\nolimits_{N\to \infty}  \tau_N(P(Y_j^{(N)})) = \tau(P(c_j))\end{equation} 
for any polynomial $P$ in the non-commuting variables $c_j,c_j^*$.
The term $*$-polynomial would be probably less abusive: Here $P(Y_j^{(N)})\in M_N$ denotes
the matrix obtained after substituting $ (Y_j^{(N)},{Y_j^{(N)}}^*)$ to $(c_j,c_j^*)$ in $P(c_j)$. Equivalently,
for almost all $\omega$ in our probability space $\Omega$, all the $\tau_N$-moments combining
the matrices $(Y_j^{(N)}(\omega))$  and their adjoints converge to the analogous $\tau$-moments
in $(c_j)$ and their adjoints.

   The above  \eqref{in2bis} from \cite{HT3}  was recently extended to unitary matrices (and a few other cases)
by  Collins and    Male, see \cite{CM}. In that case, $Y^{(N)}$ is uniformly distributed over the unitary group
$U(N)$ of all $N\times N$ unitary matrices, and $(c_j)$ has to be replaced by 
a free family of  Haar unitaries. 

If $E$ is $C$-exact, then \cite{HT3} implies
  $||| (a_1,\cdots,a_n)|||\le C \|\sum c_j \otimes a_j\|$.
  A fortiori this implies
  the following result proved in \cite{HT2} (prior to \cite{HT3}): If 
  $E$ is $C$-exact, then
\begin{equation}\label{in3} 
\forall n\ \forall (a_1,\cdots,a_n)\in E^n \quad ||| (a_1,\cdots,a_n)|||\le 2 C  \max\{\| (\sum a_j^*a_j)^{1/2} \|,\| (\sum a_ja_j^*)^{1/2}\| \}.
\end{equation}

  The starting point of our study of subexponential spaces
  is the observation (that we made several years ago after reading \cite{HT2})
  that \eqref{in3} remains valid if $E$ is $C$-subexponential.
  More precisely, we   insist on formulating universal bounds that correspond to an estimate of the speed
  of convergence in \eqref{in2bis} or \eqref{in3}. Such inequalities, that are crucial in the sequel, are implicit in \cite{HT2,HT3,HT4}.
 For instance, given $\vp>0$ there is $\gamma_\vp>0$   such that: \\ For any $E$ and
   any    $a_1,\cdots,a_n \in E$, any $\vp>0$     we have
$$\E \|\sum\nolimits_1^n Y_j^{(N)} \otimes a_j\| \le C (1+\vp)\left(2 +\gamma_\vp(\frac{\log(K_E(N,C))+1}{ N})^{1/2} \right)\max\{\| (\sum a_j^*a_j)^{1/2} \|,\| (\sum a_ja_j^*)^{1/2}\| \}.$$
This can be deduced from \cite{HT2}. We include a quick proof
based as  \cite{HT2} on the Wick formula but
taking advantage of  a result of Buchholz \cite{Buch}.\\
For our application to Grothendieck's inequality \eqref{in3} is enough. This motivates
 the definitions of ``tight" and ``completely tight" in Definition \ref{d1} below.
However,    for other $C^*$-algebraic questions, one needs the more refined  
\eqref{in2bis}  from 
 \cite{HT3}.   One can also deduce from  \cite{HT3}
 an estimate of the speed
  of convergence, but as the right hand side is more precise the error term is less well controled.
  Thus one obtains a bound of the form
   \begin{equation}\label{cu}\E \|\sum\nolimits_1^n Y_j^{(N)} \otimes a_j\| \le C (1+\vp)\left(1 +\gamma'_\vp\frac{{K_E(N,C)}^{4}}{ N} \right)  \|\sum c_j \otimes a_j\|,\end{equation}
 where $\gamma'_\vp$ may  now depend on both $\vp$ and $n$.\\
  More generally (see \eqref{htt+} below),  if we 
 replace   $\sum\nolimits_1^n Y_j^{(N)} \otimes a_j$ 
  and $\sum\nolimits_1^n c_j \otimes a_j$  by   polynomials $P(Y_j^{(N)})$ and $P(c_j)$
  of degree $d$ in the non-commuting variables $c_j,c_j^*$, then the
  analogous inequality can be deduced from \cite{HT4}, but now
  $\gamma'_\vp$ is allowed to  depend on $d$ in addition to $\vp$ and $n$.

Using the concentration of measure method, one can 
deduce from these inequalities surprisingly strong almost sure consequences,
via the following known Lemma (see \S \ref{scon} for the proof).
\begin{lem}\label{lemcon}
Consider   the following event: let $\Omega_{\vp,n}(k,N) \subset \Omega$
denote the set of $\omega\in \Omega$ such that
  \begin{equation}\label{cucu}\forall a_j\in M_k\quad \left| \|\sum\nolimits_1^n Y_j^{(N)}(\omega) \otimes a_j\| -
 \E\|\sum\nolimits_1^n Y_j^{(N)} \otimes a_j\| \right | \le \vp  \E\|\sum\nolimits_1^n Y_j^{(N)} \otimes a_j\|.
 \end{equation}
Then 
 for any $\vp>0$ there is 
a constant $c_\vp>0$ such that whenever $N\ge c_\vp n k^2$ we have
  \begin{equation}\label{cucu2}\P ( \Omega_{\vp,n}(k,N) ) \ge 1- \exp {-c'\vp^2 N}, \end{equation}
where $c'>0$ is an absolute numerical constant. 
\end{lem}
Since $\Omega_{\vp,n}(k',N)\subset \Omega_{\vp,n}(k,N)$
for any $k\le k'$ we may focus on
the largest $k$ such that $N\ge c_\vp n k^2$, i.e. on $k_{\vp,n}(N) =[( c_\vp^{-1}n^{-1}N  )^{1/2}]$
(the main point is $k_{\vp,n}(N)\approx N^{1/2}$), and rewrite \eqref{cucu2} as 
$$\P(
 \Omega_{\vp,n}(k_{\vp,n}(N)  ,N)^c ) \le   \exp {-c'\vp^2 N}.$$
 Since $\sum\nolimits_N  \exp {-c'\vp^2 N }<\infty$, 
 we find that, for any $n$ and $\vp>0$, we have
 $$\P( \liminf\nolimits_{N\to \infty} \Omega_{\vp,n}(k_{\vp,n}(N)  ,N))=1. $$ 
 In otherwords,    for almost all $\omega$, the property in
  \eqref{cucu} with    $k=k_{\vp,n}(N)    $ holds for all $N$ large enough.\\
  This explains why control of the moment as in \eqref{cu} leads to 
  rather strong almost sure consequences.
  For instance \eqref{cu} implies that if $K_E(N,1+\vp)\in o(N^{1/4})$
  for any $\vp>0$, then for almost all $\omega$
  $$\lim_{N\to \infty}\|\sum\nolimits_1^n   Y_j^{(N)} (\omega) \otimes a_j \|
  =    \|\sum c_j \otimes a_j\|.$$
  More generally, consider a sequence of  integers $K(N)$ 
  such that $K(N) \in o(N^{1/4})$, and assume that 
  we have $n$-tuples $a^{(N)} =(a^{(N)}_j) \in M_{K(N)}^n$ such that 
$$\lim_{N\to \infty}\|\sum\nolimits_1^n c_j \otimes  a^{(N)}_j    \|
  =    \|\sum c_j \otimes a_j\|.$$
Then 
  Lemma \ref{lemcon} implies that for almost all $\omega$
  $$\lim_{N\to \infty}\|\sum\nolimits_1^n Y_j^{(N)}(\omega) \otimes a^{(N)}_j \|
  =    \|\sum c_j \otimes a_j\|.$$
    We use this phenomenon in the construction of our non-exact subexponential examples.
  The $C^*$-algebra in \S \ref{s7} is the simplest to describe:
  we just consider  for $j\in \N$ the block diagonal sum
  $$u_j(\omega)=\oplus_{N\ge 1} Y_j^{(N)}(\omega)\in \oplus_{N\ge 1}  M_N$$
  and we define $A(\omega)$ as the unital $C^*$-algebra generated by $\{u_j(\omega)\mid j\in \N\}$ in  $\oplus_{N\ge 1}  M_N$.\\
  Then,  for almost all $\omega$,  $A(\omega)$ is subexponential with constant 1 but is not exact.
    
  It seems natural to wonder what becomes of  \eqref{in3} or  \eqref{in2} when $E$ is no longer assumed exact or subexponential.
  We propose some leads in this direction in \S \ref{s8} below.
   
  This paper is closely linked 
  to \cite{Pq2}  where the ``growth" of an operator space $E$ is studied via
  a different number denoted by $k_E(N,C)$.  
  There is  an obvious upper bound (for a fixed constant $C$)  $K_E(N,C)\le Nk_E(N,C)$,
so the growth of $K_E$ is dominated by that of $k_E$, but we know nothing in the converse direction. Various other questions are mentioned at the end of \S \ref{s3}. 
\section{Background on Operator Spaces}
By definition, an operator space is just a closed subspace $E\subset B(H)$ of the space of bounded operators on a Hilbert space $H$.
For any $N\ge 1$, we denote by  $M_N(E)$  the space
of $N\times N$ matrices with entries in $E$.\\ In operator space theory, the space $E$ is equipped not only with the induced norm, but also with the sequence of norms induced on $M_N(E)$ by $M_N (B(H)).$
The  space  $M_N (B(H))$ is here equipped with the norm
associated to the identification  $M_N (B(H))\simeq B(H\oplus\cdots\oplus H)$. \\
In this theory, the space $B(E,F)$ of all bounded linear maps $u:\ E \to F$
between two  operator spaces $E,F$  is replaced by
the space $CB(E,F)$ of all the completely bounded  (in short c.b.) ones, defined as follows. \\
For any  given  $N\ge 1$ we denote by
$$u_N  :\ M_N(E)\to M_N(F)$$
the mapping taking $[a_{ij}]$ to $[u(a_{ij})]$. Equivalently,
using the isomorphisms $M_N(E) \simeq M_N\otimes E$, $M_N(F) \simeq M_N\otimes F$   we will  identify $u_N$ to $Id\otimes u:\ M_N\otimes E \to M_N\otimes E$.
The mapping $u$ is called  completely bounded  (in short c.b.) if $\sup\nolimits_{N\ge 1} \|u_N\|<\infty$.
We then define $$\|u\|_{cb}=\sup\nolimits_{N\ge 1} \|u_N\|.$$  
We say that $E,F$ are completely isomorphic if there is a c.b. isomorphism
$u:\ E \to F$ with c.b. inverse.\\
Moreover, if $E,F$ are two operator spaces that are isomorphic as Banach spaces, we set
$$d_N(E,F)=\inf \{ \|u_N\| \|(u^{-1})_N\| \}$$
where the inf runs over all the isomorphisms $u:\ E \to F$.
We set $d_N(E,F)=\infty$ if $E,F$ are not isomorphic. When $N=1$ we recover the usual
Banach-Mazur distance.\\
Similarly, if $E,F$ are completely isomorphic, we set
$$d_{cb}(E,F)=\inf \{ \|u\|_{cb} \|u^{-1}\|_{cb} \}$$
where the inf runs over all the complete isomorphisms $u:\ E \to F$.\\
It is an easy exercise to check that
if $E,F$ are of the same finite dimension, we have
$$d_{cb}(E,F)=\sup\nolimits_{N\ge 1} d_N(E,F).$$
Moreover, a simple compactness argument shows that  (again if $\dim(E)=\dim(F)<\infty$)
there is an isomorphism $u:\ E \to F$
such that $\|u\|_{cb} \|u^{-1}\|_{cb} =d_{cb}(E,F)$, and after scaling
we may assume e.g. $\|u\|_{cb}=d_{cb}(E,F)$ and $\|u^{-1}\|_{cb} =1$.

Given a  bilinear form $\v: E\times F\to \C$ we define $$\v_N:\ M_N(E) \times M_N(F)\to M_N\otimes M_N\simeq M_N(M_N)$$ as 
the bilinear map
defined by $\v_N(y\otimes a, z\otimes b)=y\otimes z \ \v(a,b)$.
The form $\v$ is called (jointy) c.b. if 
$$\|\v\|_{cb}=\sup\nolimits_N\|\v_N\|<\infty.$$
The operator space dual $F^*$ of an operator space $F$
is characterized by the fact that for any $E$ and any
bilinear form $\v: E\times F\to \C$ the associated linear map $u_\v:\ E\to F^*$
satisfies
$$\|u_\v\|_{cb}=\|\v\|_{cb}.$$
The existence  (for some $\cl H$) of an isometric embedding $F^*\subset B(\cl H)$
of the Banach dual $F^*$ for which this holds is a consequence of Ruan's fundamental theorem
(see \cite{ER,P4}).\\
The following  Lemma due to Roger Smith will be very  useful.
\begin{lem}\label{rs} Let $E\subset M_K$ be any operator space.
Then for any operator space $X$ and any bounded linear map  $u:\ X\to E$
we have
$$\|u\|_{cb}=\|u_K\|.$$
\end{lem}

We refer the reader to   \cite{ER,P4} for a proof of this and for  more information on operator spaces.

 The row and column spaces 
 $R={\rm span}[e_{1j}]\subset B(\ell_2)$ and $C={\rm span}[e_{i1}]\subset B(\ell_2)$ are fundamental examples of operator spaces, as well as the finite dimensional versions:
 $$R_n={\rm span}[e_{1j},1\le j\le n]\subset M_n \qquad C_n={\rm span}[e_{i1},1\le i\le n]\subset M_n.$$
 For $a=(a_1,\cdots,a_n)$ with $a_j \in B(H)$ we denote
   \begin{equation}\label{RC}\|a\|_{RC}=
     \max\{\| (\sum a_j^*a_j)^{1/2} \|,\| (\sum a_ja_j^*)^{1/2}\| \},\end{equation}
    \begin{equation}\label{R,C}\|a\|_{R}=  \|  (\sum a_ja_j^*)^{1/2}\|  , \ {\rm and}\ 
\|a\|_{C}=   \| (\sum a_j^*a_j)^{1/2} \|,\end{equation}
 so that     \begin{equation}\label{RCC}\|a\|_{RC}=  \max\{\|a\|_{R},\|a\|_{C} \}.\end{equation}
The reader should observe that when $\dim(H)=N$ then 
$\|a\|_{R},\|a\|_{C}$   are just short cut notation for
the norms in $M_N(E)$ respectively for $E=R_n,C_n$. Indeed, we have  
$\|a\|_{R}=\|\sum  a_j\otimes e_{1j}  \|_{M_N(R_n)}$ and $\|a\|_{C}=\|\sum  a_j \otimes e_{j1}  \|_{M_N(C_n)}$. The norm
$\|a\|_{RC}$ corresponds similarly  to the span of $[e_{1j}\oplus e_{j1},1\le j\le n]$ in $R_n\oplus C_n$.

 Let $(c_j)$ be a free circular system in a von Neumann algebra $M$ equipped with a normalized trace $\tau$.
              For any $a_j\in M_k$ we have
                   \begin{equation}\label{s2}\max\{\| (\sum a_j^*a_j)^{1/2} \|,\| (\sum a_ja_j^*)^{1/2}\| \}  \le \|\sum c_j \otimes a_j\|\le 2 \max\{\| (\sum a_j^*a_j)^{1/2} \|,\| (\sum a_ja_j^*)^{1/2}\| \}.\end{equation}
      Indeed, setting $S=\sum c_j \otimes a_j$ the lower bound   follows from the identities
                $\sum a_j^*a_j \otimes 1= (Id \otimes \tau)(S^*S)$
                and $\sum a_ja_j^* \otimes 1= (Id \otimes \tau)(SS^*)$. 
                The upper bound follows from the decomposition of $c_j$ as a sum of 
                free creation and annihilation operators. See e.g. \cite{PS} for details.
                
\section{Concentration of measure and Random Matrices}\label{scon}
We will use the term ``complex valued Gaussian" random variable  for 
any random variable of the form $g=g'+ig''$ with $(g',g'')$ independent
   real valued Gaussian  variable such that $\E g'=\E g''=0$ and $\E|g'|^2= \E|g''|^2$. 
Actually all our Gaussian   variables will be  assumed to have zero mean.\\
We will denote by $Y^{(N)}$ a random $N\times N$-matrix with i.i.d. 
  complex Gaussian entries with   $L_2$-norm equal to $N^{-1/2}$ 
   and we
   denote by $(Y_j^{(N)})$ a sequence of i.i.d. copies of $Y^{(N)}$. 
   
Given an operator space $E\subset B(H)$ and $a=(a_1,\cdots,a_n)\in E^n$, we will study
the    $M_N(E)$-valued random variable  $S_a$
defined by
$$S_a=\sum\nolimits_1^n Y_j^{(N)}\otimes a_j.$$
This is a Gaussian variable with values in a Banach space,
to which the known concentration of measure inequalities,
that we now recall, can be applied.\\
Let $f: \R^d\to \R$ be a function such that
$$ \sigma= \sup \{\frac{ |f(x)-f(y)|}{  \|x-y\|_2}\mid x\not=y\in \R^d \}<\infty $$
where $\|.\|_2$ denotes the Euclidean norm on $\R^d$. Then, if $\P$ is the canonical Gaussian measure
on $\R^d$, we have
 \begin{equation}\label{bor}\forall t>0\quad \P\{ |f-\E f|>t\}\le 2 \exp{- t^2/2\sigma^2}.\end{equation}
See \cite{Led} for details. Note that it is much easier (see \cite{P02}) to prove
this with an upper bound of the form $ 2 \exp{- c t^2/\sigma^2}$ for some absolute numerical constant
$c$, and, as often, this is enough for our purposes.\\
In particular, we may view $Y^{(N)}$ as an $M_N$-valued variable 
defined on $\R^{2N^2}$ of the form
$$Y^{(N)}=\sum\nolimits_{ij}  (2N)^{-1/2}(g'_{ij}+ \sqrt{-1} g''_{ij}) \  e_{ij}.$$
Applying the above to $f=\|Y^{(N)}\|$ on $\R^{2N^2}$ 
we find $\sigma=(2N)^{-1/2}$ and
 \begin{equation}\label{bor2}\forall t>0\quad \P\{ |\|Y^{(N)}\|-\E\|Y^{(N)}\||>t\}\le 2 \exp{- N t^2 }.\end{equation}
More generally,
if we take $f=\|S_a \|$ (here the norm is in $M_N(E)$) on $\R^{2nN^2}$  we find
$$\sigma=(2N)^{-1/2}\sup \{\| \sum z_j a_j \|\mid z_j\in \C,\  \sum |z_j|^2\le 1\}.$$
Note that with the above notation \eqref{RC} we have
$$\sup \{\| \sum z_j a_j \|\mid z_j\in \C,\  \sum |z_j|^2\le 1\}\le \min\{\|a\|_{R},\|a\|_{C}\}\le \max\{\|a\|_{R},\|a\|_{C}\}=\|a\|_{RC}$$
and hence if we assume $ \|a\|_{RC}\le 1$ we find again
 \begin{equation}\label{bor2}\forall t>0\quad \P\{ |\|S_a\|-\E\|S_a\||>t\}\le 2 \exp{- N t^2 }.\end{equation}
In particular, by the classical Borel-Cantelli argument, this implies that almost surely
 \begin{equation}\label{limsup}\limsup\nolimits_{N\to \infty}  |\|S_a\|-\E\|S_a\||=0.\end{equation}
  
 \begin{proof}[Proof of Lemma \ref{lemcon}] Again let $f=\|S_a\|$.
 We first claim that
 $$(2N)^{1/2}\sigma \E\|Y^{(N)}\|  \le \E\|S_a  \|.$$
 Note that for any linear form $\xi$ such that $\|\xi\|_{E^*}\le 1$, we have
 $(Id\otimes \xi)(S_a  )=\sum\nolimits_1^n Y_j^{(N)} \xi(a_j)$, and hence
 $\E\|(Id\otimes \xi)(S_a  )\|= (\sum\nolimits_1^n |\xi(a_j)|^2)^{1/2}\E\|Y^{(N)} \| $.
 Note that for this particular choice of $f$ we have
 $\sup_{\xi\in E^*} (\sum\nolimits_1^n |\xi(a_j)|^2)^{1/2} =\sup \{\| \sum z_j a_j \|\mid z_j\in \C,\  \sum |z_j|^2\le 1\}$.
 Thus taking the supremum of $\E\|(Id\otimes \xi)(S_a  )\|$ over all $\|\xi\|_{E^*}\le 1$ we find our claim.
 Note that 
 $   \| Y^{(N)} \| \ge (\sum\nolimits_1^N     |Y^{(N)}_{1j}|^2 )^{1/2}  $ and hence
 $   \E\| Y^{(N)} \| \ge \gamma(1)$ where $\gamma(1)=\|g\|_1$ for any complex Gaussian variable $g$ normalized in $L_2$. Thus we obtain
 $$(2N)^{1/2}\gamma(1) \sigma  \le \E\|S_a  \|.$$
 Let $\cl N$ be a $\delta$-net in the unit ball
 of $M_k^n$ equipped with the norm $a\mapsto \E\|S_a  \|.$ Since $\dim_{\R}(M_k^n)=2nk^2$,
 we know (cf. e.g. \cite[p. 58]{FLM} or \cite[p. 49]{P-v}) that there is
 such a net with $|\cl N|\le (1+2/\delta)^{2nk^2}$.
 By \eqref{bor} for each $a\in \cl N$ we have
 $$\P\{ |\|S_a\|-\E \|S_a\||>\vp \E \|S_a\|\} \le 2 \exp{-\vp^2 N \gamma(1)^2   }$$
 and hence if we set
 $$\Omega'=\{\exists a\in \cl N \  |\|S_a\|-\E \|S_a\||>\vp \E \|S_a\|\}$$
 we have
 $$\P(\Omega')  \le 2 (1+2/\delta)^{2nk^2} \exp{-\vp^2 N \gamma(1)^2   }\le \exp(4nk^2/\delta -\vp^2 N \gamma(1)^2  ). $$
 To simplify, let us take $\delta=\vp<1$ and assume that $\vp^2 N \gamma(1)^2\ge 8nk^2/\delta$,
 which boils down to $8 \vp^3 nk^2\le  N$. We have then 
  $$\P(\Omega')   \le 2\exp{ -\vp^2 N \gamma(1)^2/2  } $$ and
 for any $\omega \in \Omega'$
 $$\forall a\in \cl N\quad (1-\vp) \E \|S_a\|   \le  \|S_a(\omega)\|\le (1+\vp) \E \|S_a\| $$
 but by a well known argument (see e.g.\cite[p.49]{P-v}) this implies
 for the same $\omega$
 $$\forall a\in M_k^n\quad (1-3\vp)(1-\vp)^{-1} \E \|S_a\|   \le  \|S_a(\omega)\|\le (1+\vp)(1-\vp)^{-1} \E \|S_a\| $$
 so the conclusion follows by a straightforward adjustment of $c_\vp$ and $c'_\vp$.
 \end{proof}
 \section{Operator space versions of Grothendieck's theorem}\label{s3-}
Our goal is to study a generalization of the notion of exact operator space for which the version of Grothendieck's factorization theorem obtained in \cite {PS} is still valid. 
   The latter asserts that, if $E,F$ are exact operator spaces (assumed separable for simplicity), any c.b. map from $E$ to $F^*$ factors
   through $R\oplus C$.
   The later proofs of \cite{HM2} and \cite{RV} deduce the full force of the factorization from
   an apparently weaker inequality. This motivates the
 following 
 \begin{dfn}\label{d1} Let $C\ge 1$ be a constant. We will say that an operator space $E$ is $C$-tight
 if for any $n$ and any $a_1,\cdots,a_n\in E$ we have
 $$\limsup_{N\to \infty}\E\left\|\sum\nolimits_1^n Y_j^{(N)}\otimes a_j\right\|   \le C \|a\|_{RC}.$$
 We will say that $E$ is completely $C$-tight if all the spaces $M_N(E)$ ($N\ge 1$) are $C$-tight.
 We will say that $E$ is tight (resp. completely tight) if it is $C$-tight  (resp. completely $C$-tight)  for some $C$.
\end{dfn}
With this terminology, we can refomulate the starting point of \cite{JP} like this:
\begin{lem}\label{tame} If two operator spaces $E,F$ are   respectively $C_E$-tight and $C_F$-tight
(for some constants $C_E,C_F$)
then any $u\in CB(E,F^*)$, with associated bilinear form $\v$, satisfies
for any $a_1,\cdots,a_n\in E$ and $b_1,\cdots,b_n\in F$
 \begin{equation}\label{ls0} |\sum \v(a_j,b_j)  |=|\sum \langle u(a_j),b_j\rangle  |\le \lambda \|a\|_{RC} \|b\|_{RC}, \end{equation}
 where $\lambda=C_EC_F\|u\|_{cb}$.
\end{lem}
\begin{proof} Let $S_a^{(N)}=\sum\nolimits_1^n Y_j^{(N)}\otimes a_j$
and $T_b^{(N)}=\sum\nolimits_1^n \overline{Y_j^{(N)}}\otimes b_j$.
Let $\v: E\times F \to \C $ be the c.b. bilinear form associated to $u$.
By definition of $\|\v\|_{cb}$ we have
$$\|\v_N(S_a^{(N)}, T_b^{(N)})  \|\le \|\v\|_{cb} \|   S_a^{(N)}\|_{M_N(E)}   \|   T_b^{(N)}\|_{M_N(E)}  $$
and hence
 \begin{equation}\label{ls1}\limsup_{N\to \infty}\|\v_N(S_a^{(N)}, T_b^{(N)})  \|\le \|\v\|_{cb} C_E\|a\|_{RC} C_F\|b\|_{RC}. \end{equation}
But now
$$\v_N(S_a^{(N)}, T_b^{(N)}) =\sum\nolimits_{i,j} \v(a_i,b_j) Y_i^{(N)}\otimes \overline{Y_j^{(N)}}\in M_N\otimes M_N=B(\ell_2^N\otimes \ell_2^N)$$
and the linear form $\psi : \  M_N\otimes M_N \to \C$ defined by 
$\psi (x\otimes y)=\tau_N(x {}^t y)$ has norm 1 on $
B(\ell_2^N\otimes \ell_2^N)$ (indeed its action  on $B(\ell_2^N\otimes \ell_2^N)$
can be written as  taking $T\in B(\ell_2^N\otimes \ell_2^N)$ to     $\langle T\xi,\xi\rangle$ with
$\xi =N^{-1/2}\sum e_j\otimes e_j $). Therefore
 \begin{equation}\label{ls2}|\sum\nolimits_{i,j} \tau_N( Y_i^{(N)}{Y_j^{(N)}}^*) \v(a_i,b_j)|=| \psi (\v_N(S_a^{(N)}, T_b^{(N)})) |\le |\v_N(S_a^{(N)}, T_b^{(N)}) |\end{equation}
and by   weak convergence (see \eqref{voi}) 
we have  $ \tau_N( Y_i^{(N)}{Y_j^{(N)}}^*)\to \delta_{i,j}$ a.s. when $N\to \infty$ and hence 
$\sum\nolimits_{i,j} \tau_N( Y_i^{(N)}{Y_j^{(N)}}^*)    \v(a_i,b_j)\to \sum\nolimits_{j} \v(a_j,b_j)$ a.s.
so that combining \eqref{ls1} and  \eqref{ls2} we obtain the announced \eqref{ls0}.
\end{proof}
\begin{rem} Actually the preceding uses only a weakening   of complete boundedness
called tracial boundedness, see \cite[p. 291]{P4}.
\end{rem}
\begin{rem}\label{clar} When $E,F$ are completely tight, or merely when
$M_N(E),M_N(F)$ are   tight  it is natural to try
to apply Lemma \ref{tame} to the bilinear form $\v_N$.
The natural analogous assumption is then
for any $a_1,\cdots,a_n\in M_N(E)$ and $b_1,\cdots,b_n\in M_N(F)$
 \begin{equation}\label{ls00} \|\sum \v_N(a_j,b_j)  \|_{M_N\otimes_{\min} M_N}\le \lambda \|a\|_{RC} \|b\|_{RC}, \end{equation}
 Equivalently, let $\xi,\eta$ be arbitrary
 in  the unit ball of $\ell_2^N\otimes_2 \ell_2^N$ 
  and let us denote by $\psi_{\xi,\eta}:\  M_N\otimes_{\min} M_N \to \C$
  the linear form defined by $\psi_{\xi,\eta}(T)=\langle T \xi,\eta\rangle$.
  Note that $\|T\|=\sup\{ |\psi_{\xi,\eta}(T)|\mid \|\xi\|_2\le 1,\|\eta\|_2\le 1\}$.
  Therefore, assuming \eqref{ls00} is the same as assuming that
  all the scalar valued bilinear forms $\psi_{\xi,\eta} \v_N=\psi_{\xi,\eta}\otimes  \v:\ M_N(E)\times M_N(F) \to \C$ (with $\|\xi\|_2\le 1,\|\eta\|_2\le 1$) all satisfy
  \eqref{ls0}.\\
  Moreover we have
 \begin{equation}\label{ls33}\|\v_N\|=\sup\{ \|\psi_{\xi,\eta}\otimes \v\|\mid \|\xi\|_2\le 1,\|\eta\|_2\le 1\}.\end{equation} 
  \end{rem}
The following statement generalizes the main result of \cite{PS}.
This new formulation became clear after \cite{HM2} appeared.
Indeed,  although \cite{HM2}  does not consider it, Mikael  de la Salle
and the author 
(see   the second proof given in \cite[\S 18 p. 303]{Pgt})
adapted their method to prove essentially the same as the next result.
However, more recently Regev and Vidick gave a strikingly simple proof
of the same step. We recommend their paper \cite{RV} to the interested reader.
  \begin{thm}\label{clari} Let $E,F$ be arbitrary operator spaces,   let $\v:\ E\times F \to \C$ be a bilinear form (associated to
  $u:\ E\to F^*$), and let $\lambda>0$ be any constant.
  Assume that  $\v_N$ satisfies \eqref{ls00}  for all $N\ge 1$. Then 
   for any $a_1,\cdots,a_n\in E$ and $b_1,\cdots,b_n\in F$ and  $t_1>0,\cdots,t_n>0$  
 \begin{equation}\label{ls3}|\sum \langle u(a_j),b_j\rangle  |\le \lambda (\|a\|_{R} \|b\|_{C}+
   \|({t_j} a_j)\|_{C} \|({t_j}^{-1}b_j)\|_{R} ). \end{equation}
\end{thm}
\begin{rem} Actually, this still holds if   the assumption \eqref{ls00} imposed on  $\v_N$  is weakened to
\begin{equation}\label{ls0bbis}  \|\sum \v_N(a_j,b_j)  \|_{M_N\otimes_{\min} M_N}\le \lambda (\|a\|^2_{R}+\|a\|^2_{C})^{1/2}(\|b\|^2_{R}+\|b\|^2_{C})^{1/2}  .\end{equation}
When $E,F$ are $C^*$-algebras, \eqref{ls0bbis} and hence \eqref{ls3}  is satisfied with $\lambda=\|u\|_{cb}$ and this is sharp (see \cite{HM2}
or \cite{Pgt} for details).
\end{rem}
  \begin{cor} If two operator spaces $E,F$ are   respectively completely $\hat C_E$-tight and 
  completely $\hat C_F$-tight,
(for some constants $\hat C_E,\hat C_F$) then any $u\in CB(E,F^*)$ satisfies
for any $a_1,\cdots,a_n\in E$ and $b_1,\cdots,b_n\in F$ and  $t_1>0,\cdots,t_n>0$  
 \begin{equation}\label{ls3bis}|\sum \langle u(a_j),b_j\rangle  |\le \hat C_E \hat C_F \|u\|_{cb} (\|a\|_{R} \|b\|_{C}+
   \|({t_j} a_j)\|_{C} \|({t_j}^{-1}b_j)\|_{R} ). \end{equation}
\end{cor}
 \begin{proof} Let $\v:\ E\times F \to \C$ be the bilinear form associated to $u$. Let $N\ge 1$.
We will invoke Remark \ref{clar} and \eqref{ls33}.
 Let $\lambda= \hat C_E \hat C_F \|u\|_{cb}$. Since $M_N(E),M_N(F)$ are   tight, 
 by Lemma \ref{tame} the mappings $\psi_{\xi,\eta}\otimes \v:\ M_N(E)\times M_N(F) \to \C$
 all satisfy \eqref{ls0} and hence (see Remark \ref{clar})  $\v_N$ satisfies \eqref{ls00},
 and the Corollary follows from the preceding Theorem.  
 \end{proof}
By the same method as in \cite{PS} (see \cite[Prop. 18.2]{Pgt}), this implies
 \begin{cor}\label{cor1} If $E,F$ are   respectively completely $\hat C_E$-tight and completely $\hat C_F$-tight,
then any $u\in CB(E,F^*)$  admits, for some Hilbert spaces $H,K$, a factorization of the form
$$E {\buildrel v  \over \longrightarrow} H_r \oplus K_c {\buildrel   w \over \longrightarrow}  F^* $$
 with $\|v\|_{cb} \|w\|_{cb} \le 2 \hat C_E \hat C_F \|u\|_{cb}$. Here the spaces $K_c=B(\C,K)$
 and  $H_r=B(\bar H,\C)$ are equipped with their natural operator space structure
 and the direct sum $H_r \oplus K_c $ is taken in the block diagonal sense.
 When $H=\ell_2$ (resp. $K=\ell_2$) we have $H_r=R$ (resp. $K_c=C$).
\end{cor}
  \section{Subexponential operator spaces}\label{s3}
Let $E$ be a finite dimensional operator space. Fix $C>0$.
We denote by $K_E(N,C)$ the smallest integer $K$ such that
there is an operator subspace $F\subset M_K$
such that
$$d_N(E,F)\le C.$$

We will say that an operator space $X$ is   $C$-exact 
if for any finite dimensional subspace $E\subset X$
 there is a $K$ and $F\subset M_K$
such that $d_{cb}(E,F)\le C$. We denote by $ex(X)$ the infimum of such $C$'s.
We say that $X$ is exact if it is $C$-exact for some $C\ge1$.
As shown by Kirchberg, a $C^*$-algebra $X$ is exact iff $ex(X)=1$ .
We do not know whether the analogue of this for subexponential $C^*$-algebras  is true.
See \cite[ch.17]{P4} or \cite{BO} for more background on exactness
(note however that our definition of   $C$-exact is not quite the same
as in \cite{P4} where $C$-exact means $ex(X)\le C$).

\begin{lem} An operator space $X$ is $C$-exact iff
$$   \sup\nolimits_{N\ge 1} K_E(N,C)<\infty.$$
 for any finite dimensional subspace $E\subset X$.
\end{lem}
\begin{proof} The only if part is obvious since $d_N\le d_{cb}$.
Conversely, assume that for some fixed $K$ we
have
$$  \sup\nolimits_{N\ge 1} K_E(N,C)\le K.$$
We have then for each $N$ a subspace $F_N\subset M_K$ and  a mapping
$u(N):\ E \to F_N$ such that 
$\|u(N)_N\|\le C$ and $ \|{u(N)^{-1}}_N\|\le 1$.
Let $F$ be an ultraproduct of $(F_N)$ along a free ultrafilter (see e.g. \cite{P4}
for ultraproducts of operator spaces),
and let $u:\ E\to F$ be the mapping associated to $(u(N))$.
Then clearly $\|u\|_{cb}\le C $ and  $\|u^{-1}\|_{cb}\le 1$. 
So we obtain $d_{cb}(E,F)\le C$ and $F$ obviously embeds completely isometrically into $M_K$.
\end{proof}

\begin{dfn} We say that an operator space $X$ is $C$-subexponential
if $$\limsup_{N\to \infty} \frac{\log K_E(N,C)} {N}=0,$$
for any finite dimensional subspace $E\subset X$.
We say that $X$ is    subexponential
if it is $C$-subexponential for some $C\ge1$.\\
{\bf Note:} If $X$ itself is finite dimensional, it suffices to consider $E=X$.\\
 We will denote by $C(X)$ the infimum of the $C$'s such that 
$X$ is $C$-subexponential.
\end{dfn}

In \cite{Pq2} we introduce the following variant of $K_E(N,C)$:
\\
We denote by $ k_E(N,C) $ the smallest integer $k$ such that   there is a subspace $F$ of $M_N \oplus\cdots \oplus M_N$ (with $M_N$ repeated $k$-times)
        such that $d_N(E,F)\le C$. Obviously we have
        $$K_E(N,C) \le N k_E(N,C).$$
        We observe in  \cite{Pq2}  that for any $E$ we have
       $$  k_E(N,C)\le (\frac{3C}{C-1})^{2nN^2}   .$$
       The proof is an easy argument involving the cardinal of a $(C-1)$-net in the unit ball
       of the space $M_N(E)$, the $\R$-dimension of which is ${2nN^2} $.\\
        This implies that
        \begin{equation}\label{f}\frac{\log K_E(N,C)} {nN^2}\le C',\end{equation}
        where $C'$ depends only on $C>1$.\\
Unfortunately, we are unable to improve the last bound for general $n$-dimensional spaces $E$. More precisely,
we do not know whether there exists spaces $E$ for which the growth of ${\log K_E(N,C)}$ is intermediate between
$o(N) $ (i.e. the subexponential case) and $O(N^2)$ (the general case).

\begin{rem}\label{k}
It is easy to check that if $X$ is $C$-subexponential, the minimal tensor product 
$K(\ell_2)\otimes_{\min} X$ 
(of $X$ with the set $K(\ell_2)$ of all  compact operators on $\ell_2$)
is also $C$-subexponential. 
Indeed, by a perturbation argument we may restrict to finite dimensional subspaces
of the form $M_n(E)$ with $E\subset  X$. Then we have obviously
 \begin{equation}\label{eq3+++} K_{M_n(E)}(N,C)\le n K_E(nN,C), \end{equation} and hence the subexponential character is preserved.
\end{rem}

   The following result follows from the main estimates in \S 2 in \cite{HT2}, but we take advantage
   of \cite{Buch} to formulate an improved inequality for which we can
   give a quick sketch of proof (inspired from ideas in \cite{HT2} and  \cite{Buch}). 
  
 \begin{thm}[\cite{HT2}]\label{ht} 
 Let $p\ge 2$ be any even  integer. For any $a=(a_1,\cdots,a_n) \in B(H)^n$
 such that ${\rm tr} |a_j|^{p}<\infty$ for $1\le j\le n$, let $S_a=\sum_1^n Y_j^{(N)} \otimes a_j$.
 We have
 \begin{equation}\label{eq9.1} \E ({\rm tr}\times \tau_N)|S_a|^{p}\le \E\tau_N |  Y^{(N)} |^{p}  \max\left\{{\rm tr} (\sum a_j^*a_j)^{p/2}  , {\rm tr}(\sum a_ja_j^*)^{p/2} \right\}.\end{equation}   \end{thm}
  \begin{proof}  Let $p\ge 2$ be any even integer.
  Let $P_2(p)$ denote the set of all partitions of $[1,\ldots, p]$ into subsets each with exactly 2 elements. So an element $\nu$ in $P_2$ can be described as a collection of disjoint (unordered) pairs $\{k_j,\ell_j\}$ $(1\le j\le n)$ with $k_j\ne \ell_j$ such that $\{1,\ldots, 2n\} = \{k_1,\ldots, k_n, \ell_1,\ldots, \ell_n\}$. 
Let $X=(X_j)$ ($1\le j\le p$) be a Gaussian sequence of real valued  random variables
(i.e. all their linear combinations are Gaussian). We first recall the classical Wick formula:
$$\E(X_1\cdots X_p)=\sum\nolimits_{\nu\in P_2(p)} \prod \langle X_{k_j}, X_{\ell_j} \rangle $$
where   the product runs
over all the blocks $\{k_j,\ell_j\}$ ($j=1\cdots p/2$) of $\nu$, and the 
  scalar products are meant in $L_2$.\\
  To lighten the notation we set
  $x_j=Y_j^{(N)}$. Then, if one develops the product and the trace,  it is not hard to deduce from the Wick formula
   that there is
  a function ${\psi}\colon \ P_2(p)\to {\bb C}$  such that for any $k_1,\ldots, k_{p}$ we have
 
\[
\E \tau_N(x^*_{k_1}x_{k_2}x_{k_3}\ldots x^*_{k_{p-1}}x_{k_{p}}) = \sum_{\nu\sim (k_1,\ldots, k_{p})} {\psi}(\nu)
\]
where the notation $\nu\sim (k_1,\ldots, k_{p})$ means that $k_i=k_j$ whenever the pair $\{i,j\}$ is a block of the partition $\nu$.\\
Note that, for each $k$, taking  the $k_j$'s all equal to $k$, this implies 
\begin{equation}\label{eq9.1b} \E\tau_N( |x_k|^p)=\sum\nolimits_{\nu\in P_2(p) }  {\psi}(\nu).
\end{equation}
We may complete the proof without spelling out the precise formula for $\psi(\nu)$ but we need to note
 that $\psi(\nu)\ge 0$ and  that  the only $\nu$'s for which  ${\psi}(\nu)\not=0$
are those with all blocks formed of an odd and an even index.
We have $|S_a|^{p}=( S^*_aS_a )^{p/2}$ and hence
\[
\E ({\rm tr}\times \tau_N)|S_a|^{p}= \sum_{k_1,\ldots, k_{p}} \ \sum_{\nu\sim (k_1,\ldots, k_{p})} {\psi} (\nu) 
{\rm tr} (a_{k_1}^*     a_{k_2}  \cdots  a_{k_{p-1}}^*   a_{k_{p}}).
\]
Therefore
$$  \E ({\rm tr}\times \tau_N)|S_a|^{p}\le   \sum\nolimits_{\nu\in P_2(p)} {\psi} (\nu)  {\rm tr}  (a_\nu)\le (\sum\nolimits_{\nu\in P_2(p)} {\psi} (\nu)) \max_{\nu} {\rm tr}  (a_\nu) $$
where
$a_\nu=\sum_{ (k_1,\ldots, k_{p})\sim \nu} (a_{k_1}^*     a_{k_2}  \cdots  a_{k_{p-1}}^*   a_{k_{p}}).$
Now by a nice iteration argument of the Cauchy-Schwarz inequality,
for which the reader can find details in \cite{Buch}, one can show that the terms ${\rm tr}  (a_\nu)$ are maximal
when $\nu$ is   
either $\{1,2\},\{3,4\}\cdots\{p-1,p\}$ or 
$\{p,1\},\{2,3\}\cdots\{p-2,p-1\}$  (i.e. a partition in cyclically consecutive pairs), in which case by the trace property we have either
$ {\rm tr}(a_\nu)= {\rm tr} (\sum a_j^*a_j)^{p/2}$ or ${\rm tr}(a_\nu)= {\rm tr} (\sum a_ja_j^*)^{p/2}$. Thus,
by \eqref{eq9.1b} we obtain \eqref{eq9.1}.
 \end{proof}
      Note that \eqref{eq9.1}  is obviously best possible. It can be interpreted as 
   a sort of  ``Khintchine inequality"  for Gaussian random matrices with best possible constant.
 
   We will use the following direct consequence of Theorem \ref{ht} and concentration of measure.
     \begin{cor} For any $\vp>0$, there is a constant $\gamma_\vp$
     such that for any integer $k$  and any $a_1,\cdots,a_n \in M_k$ we have 
     \begin{equation}\label{eq30}\E\left \|   \sum\nolimits_1^n Y_j^{(N)} \otimes a_j\right\|\le (1+\vp)\left(2 +\gamma_\vp(\frac{\log(k)+1}{ N})^{1/2} \right)\max\{\| (\sum a_j^*a_j)^{1/2} \|,\| (\sum a_ja_j^*)^{1/2}\| \}.\end{equation} 
  \end{cor}
 \begin{proof} 
   Let $X$ be any Gaussian  random variable with values in a (real) Banach space B.
   Let $$\sigma(X)=\sup\{ ( \E |\xi(X)|^2)^{1/2}\mid \xi \in B^*, \|\xi\|\le 1\}$$
 It will be convenient to use the following  
  concentration of measure inequality  (see   \cite{P02} for a very simple proof): 
 $$\|\|X\| -\E   \|X\| \|_p\le (\pi/2) \sigma(X) \|g\|_p,$$
 where $g$ is a standard Gaussian normal random variable, and this implies
  \begin{equation}\label{eq2}(\E \|X\|^p)^{1/p}\le \E\|X\| +  (\pi/2) \sigma(X) \|g\|_p.\end{equation}
 We  will view $Y^{(N)}$ as $B$-valued with $B=M_N$ considered as a real Banach space.
 We have then   
 $$\sigma(Y^{(N)})\le  N^{-1/2}.$$ 
 We   denote $\|Y^{(N)}\|_p=({\rm tr}(|Y^{(N)}|^p)^{1/p}$. 
 Thus the preceding inequality applied to $X=Y^{(N)}$  yields
  by   \eqref{eq2}  
  $$ (\E\|  Y^{(N)} \|^{p}_{p})^{1/{p}}\le  (N)^{1/p} (\E\|  Y^{(N)} \|^{p})^{1/{p}}
  \le (N)^{1/p}  (  \E\|Y^{(N)}\| +  (\pi/2) \sigma(Y^{(N)}) \|g\|_{p}).$$
  It is well known   that $\lim_{N\to\infty} {\bb E}\|Y^{(N)}\|_{M_N}=2$. In fact we need only an upper bound, so we set
  $$\vp(N)=\E\|Y^{(N)}\|-2\quad {\rm and\ we\ note}\quad \lim\nolimits_{N\to\infty}\vp(N)= 0.$$
Since there is a constant $\beta$ such that $\|g\|_{p}\le \beta\sqrt{p}$ for all $p\ge 1$, we find
  $$ (\E\|  Y^{(N)} \|^{p}_{p})^{1/{p}}\le   (N)^{1/p}  (  2 +\vp(N) + \beta (\pi/2)   ({p}/N)^{1/2}).$$
 Let $S=S_a$. We again denote $\|S\|_{p}= ({\rm tr}(|S|^p)^{1/p}$ (but this time the trace is on $M_N\otimes M_k$). We   have obviously $ \|S\|\le   \|S\|_{p}
$ for any $p\ge 1$ and hence
 $$(\E \|S\|^{p})^{1/p}\le (\E \|S\|^{p}_{p})^{1/p}.$$
 By homogeneity we may assume  $\max\{\| (\sum a_j^*a_j)^{1/2} \|,\| (\sum a_ja_j^*)^{1/2}\| \} \le 1$.
 By \eqref{eq9.1} this gives us
\begin{equation}\label{eq3-}(\E \|S\|^{p})^{1/p}\le (kN)^{1/p}  (  2 +\vp(N)   + \beta (\pi/2)   ({p}/N)^{1/2}). \end{equation}
  Fix $0<\vp\le 1$. For a suitably chosen $N(\vp)$, we have for all $N\ge N(\vp)$ 
  \begin{equation}\label{eq3-bis}(\E \|S\|^{p})^{1/p}\le (kN)^{1/p}  (  2 +\vp/2   + \beta (\pi/2)   ({p}/N)^{1/2}). \end{equation}We can choose the even integer $p$ large enough so that
 $ (kN)^{1/p} \approx 1+\vp/2$:  Indeed, by taking say $p= 2[ c \vp^{-1} (\log(kN)+1)]$ and adjusting 
 the positive constant $c$ and $N(\vp)$
 we can obtain  $ (kN)^{1/p} \le 1+\vp/2$. Then,
 for some numerical constant $\beta'$, we obtain
  \begin{equation}\label{eq3}\E \|S\| \le (\E \|S\|^{p})^{1/p}\le  (1+\vp/2) (  2    +\vp /2+\beta'    (\vp^{-1} \log(kN) /N)^{1/2}),
  \end{equation}
  and this leads to \eqref{eq30}, at least for all $N\ge N(\vp)$, with $N(\vp)$ depending only on $\vp$. \\
  But for $N\le N(\vp)$ it is easy to choose  the constant $\gamma_\vp\ge N(\vp)$  to make sure that 
   \eqref{eq30} remains true. Indeed, for some $\beta_\vp$
   we have $\vp(N)\le \beta_\vp$ for all $N\ge N(\vp)$.
 \end{proof}
   \begin{rem} By   \cite[Th. 3.3]{HT2}, with the same notation as in the above Theorem \ref{ht}, 
 assuming  $\max\{\| (\sum a_j^*a_j)^{1/2} \|,\| (\sum a_ja_j^*)^{1/2}\| \}\le 1$
 we have
 for any $0\le t\le N/2$
$$   \E \exp t\|S\|^2 \le kN \exp(4t+4t^2/N) .$$
By convexity this implies
$\exp t (\E\|S\|)^2 \le kN \exp(4t+4t^2/N) ,$ and hence taking the log we find
$$\E\|S\|\le 2 \left(  1+t/N+  (4t)^{-1} \log(kN) \right)^{1/2}$$
from which taking $t=[\vp N]$ for $\vp<1/2$ it is easy to deduce \eqref{eq30}.
Note however that the deduction of \cite[Th. 3.3]{HT2} from
Prop. 2.5 and Prop. 2.7    in \cite{HT2}
  involves rather heavy calculations,  and that explains why we presented
  the above short cut (based only on Prop. 2.5 and Prop. 2.7    in \cite{HT2})
  using concentration of measure instead of invoking  \cite[Th. 3.3]{HT2}. In addition,
  this route allows us to draw the reader's attention to Buchholz's nice contribution \cite{Buch}.
    \end{rem}
     \begin{cor}\label{eq10} For any finite dimensional operator space $E$ and any $a_1,\cdots,a_n \in E$ we have
$$\E \|\sum\nolimits_1^n Y_j^{(N)} \otimes a_j\| \le C (1+\vp)\left(2 +\gamma_\vp(\frac{\log(K_E(N,C))+1}{ N})^{1/2} \right)\max\{\| (\sum a_j^*a_j)^{1/2} \|,\| (\sum a_ja_j^*)^{1/2}\| \}.$$
  \end{cor}
   \begin{proof}
    Consider $u:\ E \to F$ with $F\subset M_k$,  
  $k=K_E(N,C)$ and $ \|u_N\| \|{u^{-1}}_N\|\le C$.\\
  By homogeneity we may assume $\max\{\| (\sum a_j^*a_j)^{1/2} \|,\| (\sum a_ja_j^*)^{1/2}\| \}= 1$.
Let $b_j=u(a_j)$.  We may assume $n\le N$. Then  we have $$\max\{\| (\sum b_j^*b_j)^{1/2} \|,\| (\sum b_jb_j^*)^{1/2}\| \}  \le \|u_n\|\le \|u_N\|,$$ and also
  $$ \|\sum\nolimits_1^n Y_j^{(N)} \otimes a_j\| \le \|{u^{-1}}_N\|   \|\sum_1^n Y_j^{(N)} \otimes b_j\|.$$
  By \eqref{eq30} (applied with $b_j$ in place of $a_j$) this gives us
 $$\E \|\sum\nolimits_1^n Y_j^{(N)} \otimes a_j\| \le \|{u^{-1}}_N\| \|u_N\|  (1+\vp)\left(2 +\gamma_\vp(\frac{\log(k)+1}{ N})^{1/2}\right) ,$$
 and the result follows.
 \end{proof}
This leads us immediately to
 \begin{thm}\label{t6} Any $C$-subexponential operator space is completely $2C$-tight.
 \end{thm}
  \begin{proof} By Remark \ref{k} it suffices to show
   that $E$ is $2C$-tight. Then,  letting $N\to \infty$, the result follows
   from the preceding Corollary.
 \end{proof}
 For emphasis, we state the following immediate consequence of Corollary \ref{cor1}.
  \begin{cor}\label{cor2} If $E,F$ are   both subexponential,
then any $u\in CB(E,F^*)$  admits, for some Hilbert spaces $H,K$, a factorization of the form
$$E {\buildrel v  \over \longrightarrow} H_r \oplus K_c {\buildrel   w \over \longrightarrow}  F^* $$
 with $\|v\|_{cb} \|w\|_{cb} \le 4 C(E)  C(F)$. \end{cor}
Using Oikhberg's result  as in \cite[p. 296]{Pgt} we obtain
  \begin{cor}\label{cor3} If $E,E^*$ are   both subexponential,
then  for some Hilbert spaces $H,K$,  $E$ must be   completely  isomorphic
to $ H_r \oplus K_c$. If $E$ is separable and infinite dimensional
then $E$ must be   completely  isomorphic to either $R,C$ or $R\oplus C$.
 \end{cor}
  \begin{rem} Perhaps the preceding statement is best appreciated for a finite dimensional
  space $E$. Roughly if $E$ is not close to a space of the form $ H_r \oplus K_c$ (which in practise is easy to see since very few spaces are like that), then either $K_E(N,C)$
  or $K_{E^*}(N,C)$ must grow superexponentially.
  This can be viewed as analogous to the Figiel-Lindenstrauss-Milman estimate of the number of faces and vertices of a polytope in \cite[Th. 3.4]{FLM}.
    \end{rem}

  \begin{rem} The preceding proofs  suggest that perhaps
       one should keep track of the dependence in $E$ in studying  spaces like subexponential ones.
       One possibility would be to define $X$ as $(C,C')$-subexponential
       if for any finite dimensional $E\subset X$ we have $$\limsup_{N\to \infty} N^{-1}{\log K_E(N,C)}  \le C'.$$
       Note however that the constant $C'$ does not seem to behave as well as $C$
      (see \eqref{eq3+++}) when one passes from $E$ to $M_n(E)$.
   \end{rem}

     \begin{rem} Given an operator space $X$, it is natural to introduce the following parameter:
    $$K_X(N,C;d)=\sup\{ K_E(N,C) \mid E\subset X,\ \dim(E)=d\}.$$
       We will say that $X$ is     uniformly subexponential
     if there is $C$ such that 
    $$\forall d\ge1\quad   \limsup_{N\to \infty} \frac{\log K_X(N,C;d)} {N}=0.$$
    Similarly we will say that $X$ is     uniformly  exact if there is $C$ such that  
$$\forall d\ge1\quad\sup\{ K_X(N,C;d) \mid  N\ge 1\}<\infty.$$

It is   easy to check that if $X$ is     uniformly  exact (resp.  uniformly subexponential) then all ultrapowers of $X$ are exact
  (resp.   subexponential).
  Note however (I am indebted to Yanqi Qiu for conversations on this) that the converse is unclear.\\
  For example, $R$ or $C$  and  $R\oplus C$ are uniformly  exact.
  More generally, let $A$ (resp. $(\Omega,\mu)$) be
  any commutative
  $C^*$ algebra (resp. any measure space),  then $A$ (or $L_\infty(\Omega,\mu)$)
   and   any space of the form $A \otimes_{\min} M_N$  
  (with $N$ fixed) or $L_\infty(\Omega,\mu ; R\oplus C )$  is
 uniformly  exact.    There seem to be rather few such spaces.  It would be interesting to
  characterize   them. \end{rem}
 \begin{rem} It is tempting to weaken the definition of  subexponential spaces by replacing the limsup
    there by a liminf. Such spaces could be called
    weakly subexponential. We do not know whether this is a true weakening.
    Then  
    Corollary \ref{cor2}  
    extends to the case when one of $E,F$ is weakly subexponential and the other one
     subexponential.
    Note however
    that, a priori, the case when both   $E,F$ are   weakly subexponential is unclear.      \end{rem}    
 \section{Large constants of subexponentiality}
       We will now examine some examples. It turns out 
that the most commonly known    non-exact operator spaces are also not subexponential,
and the associated constants have a similar growth.
 
 We start by discussing maximal operator spaces.
(See e.g. \cite{P4} for the definitions of minimal and maximal operator spaces.)
 \begin{pro} Let $E$ be any $n$-dimensional   space with its maximal operator space structure. 
 Then
 $$C(E) \ge c\sqrt{n}$$
 where $c>0$ is a constant independent of $n$. \end{pro}
 \begin{proof} We transplant from exact to subexponential an argument from \cite{JP}.
 Note that $C(E^*)=1$ since $E^*$ is a minimal operator space.
 By Theorem \ref{t6} and Lemma \ref{tame} (with $F^*=E$ and $u$ the identity of $E$) we have for all finite sequences $(a_j,b_j)$ in $E\times E^*$
 $$|\sum \langle a_j,b_j\rangle |  \le  4 C(E) \|a\|_{RC}\|b\|_{RC}$$
 but here $\|b\|_{RC}=\sup\{ ( \sum |b_j(x)|^2)^{1/2}\mid x\in E, \|x\|\le 1\}$ and
 $\|a\|_{RC}\le (\sum \|a_j\|^2)^{1/2}$, so this implies
 $$|\sum \langle a_j,b_j\rangle |  \le  4 C(E)(\sum \|a_j\|^2)^{1/2}\sup\{ ( \sum |b_j(x)|^2)^{1/2}\mid x\in E, \|x\|\le 1\}$$
 and hence
 $$(\sum \|b_j\|^2)^{1/2} \le  4 C(E)\sup\{ ( \sum |b_j(x)|^2)^{1/2}\mid x\in E, \|x\|\le 1\}.$$
Equivalently, this means the $2$-summing norm $\pi_2(E)$ of the identity of $E$ is $\le  4 C(E)$.
But it is well known (see e.g. \cite[p. 35]{P-v}) that $\pi_2(E)=\sqrt n$. Thus we conclude $C(E) \ge \sqrt n/4$. 
\end{proof}
 \begin{rem}\label{rr1}  In the converse direction,  for any $n$-dimensional  operator  space $E$ we have $C(E)\le ex(E)$ and  it is known
 (see \cite[Cor. 7.7 p. 133]{P4}) 
  that $ex(E)\le  \sqrt n$.  \end{rem}
    
       \begin{rem}\label{r0}  
       We claim that $$n^{1/4}/2 \le C(OH_n)\le n^{1/4}.$$
       Indeed, applying  Corollary  \ref{eq10} with $E=OH_n$ (see \cite[\S 7]{P4}) and with $a_j$ an orthonormal basis
       we find
        $$\limsup\nolimits_{N\to \infty}  \E \left( \left\|\sum_1^n Y_j^{(N)} \otimes \overline{Y_j^{(N)}}\right\|^{1/2} \right)   \le 2 C(OH_n) n^{1/4} ,$$
       and since $ \left\|\sum_1^n Y_j^{(N)} \otimes \overline{Y_j^{(N)}}\right\| \ge \sum_1^n {\tau_N}|Y_j^{(N)}|^2$ and  (by the  law of large numbers)
       $\sum_1^n {\tau_N}|Y_j^{(N)}|^2\approx n$ we obtain
      $$n^{1/2}=\limsup\nolimits_{N\to \infty}  \E \left( (\sum\nolimits_1^n {\tau_N}|Y_j^{(N)}|^2)^{1/2} \right)   \le 2 C(OH_n) n^{1/4}$$ and hence
       $ C(OH_n) \ge n^{1/4}/2.$
       In the converse direction, we have $C(OH_n)\le ex(OH_n)$ and  it is known (see \cite[(10.8) p. 219]{P4}) that $ex(OH_n)\le n^{1/4}$.      
   \end{rem}
     \begin{rem}  We claim that $$n^{1/2}/2 \le C(R_n+C_n)\le n^{1/2}.$$
       Indeed, applying  Corollary  \ref{eq10}  with $E=R_n+C_n$ (see \cite[\S 2.7]{P4}) and with $a_j$ an orthonormal basis
       we find similarly (since $\|(a_j)\|_{RC}=1$)  $$n^{1/2}=\limsup\nolimits_{N\to \infty}  \E \left( (\sum\nolimits_1^n {\tau_N}|Y_j^{(N)}|^2)^{1/2} \right)   \le 2 C(R_n+C_n)   ,$$
       and thus we obtain
       $$ C(R_n+C_n) n^{-1/2}\ge 1/2.$$
       In the converse direction, by Remark \ref{rr1} we have $C(R_n+C_n)\le n^{1/2}.$    
   \end{rem}
    \section{More growth estimates}
 In \cite{Pm} the parameter denoted below by $n(E,c)$ was introduced for an $n$-dimensional Banach space $E$ and a constant $c$ (in \cite{Pm}  we fixed $c=2$).
    We denote   by $n(E,c)$ the smallest $k$ such that
   $E$ can be embedded $c$-isomorphically into $\ell_\infty^k$. In Banach space theory Gaussian random variables can be used to give a quick
   proof of the fact that if either $E=\ell_2^n$ or $E=\ell_1^n$ then there is $\delta= \delta_c>0$ such that $n(E,c)\ge \exp(\delta n). $ In  \cite{Pm} 
   an estimate  due to Maurey is presented showing that this  superexponential behaviour remains true
   (with $\delta= \delta(c,c')>0$) whenever $E^*$ has type $p>1$ with constant at most $  c'$.
   Incidentally, it remains an important open question whether this is true assuming only that
   $E$ has   cotype $q<\infty$  with constant at most $  c'$. \\
  The problem of estimating the number $ k_E(N,C) $ 
          (as defined just before \eqref{f}) is entirely analogous to the one considered in \cite{Pm} for $n(E,c)$. More precisely,  
                 we have simply     $ n(E,C)=k_E(1,C)$.

   The preceding inequality \eqref{eq30} allows us, in the next Lemma, to prove analogous results, for some operator spaces. Note  
   that in the Banach space case (equivalently in the case $N=1$), and taking say $c=2$, we also know that $n(E,2)\le \exp(\delta' n) $ for some universal constant $\delta'$,
  while   \eqref{f}  tells us that  for   any $n$-dimensional operator space   $E$ 
  we have $     {  K_E(N,2)}\le \exp(\delta'  {nN^2})$. Unfortunately we do not know whether this upper bound can be improved to match the lower bound
  appearing below in \eqref{eq0}.  
     \begin{lem}\label{lem15} If $E$ is    $ \ell_1^n$  equipped with its maximal operator space structure, 
     then for any $C>1 $ there are  an integer $n_0$ and   $\delta >0$   depending only on $C$ such that  for any $n\ge n_0,N\ge 1$  we have     \begin{equation}\label{eq0}K_E(N,C) \ge \exp \delta Nn   \end{equation}
     If $E=R_n+C_n$ or  $ \ell_2^n$ equipped with its maximal operator space structure
      (resp. $E=OH_n$), this still holds
     (resp.  we have $ K_E(N,C) \ge \exp \delta Nn^{1/2}$)
      for all $N\ge n$.
       \end{lem}
\begin{proof} With  the notation in Theorem \ref{ht},  let $a_j$ be the canonical basis of $ \ell_1^n$ (resp. $OH_n$, rresp.  $R_n+C_n$) and let $S=S_a$.
     Then it is easy to check on the one hand that $\E\|S\|\ge \alpha n$   (resp. rresp. $\E\|S\|\ge \alpha n^{1/2}$) for some $\alpha >0$.
     On the other hand $\|a\|_{RC}=\max\{\| (\sum a_j^*a_j)^{1/2} \|,\| (\sum a_ja_j^*)^{1/2}\| \}    $ is equal to $n^{1/2}$ (resp. $n^{1/4}$, rresp.  $1$).
     Thus by Corollary \ref{eq10}
  we find if $E= \ell_1^n$
     $$\alpha n \le C (1+\vp)\left(2 +\gamma_\vp(\frac{\log(K_E(N,C))+1}{ N})^{1/2} \right)n^{1/2},$$
     from which we deduce for $n$ large enough
     $$(\alpha/C (1+\vp)) n^{1/2} \approx (\alpha/C (1+\vp)) n^{1/2} -2     \le \gamma_\vp(\frac{\log(K_E(N,C))+1}{ N})^{1/2},$$
     which is the announced lower bound (taking e.g. $\vp=1$). The cases  $E= OH_n$ and $E=R_n+C_n$ are similar.
      When $(a_j)$ is the basis of $ E=\ell_2^n$  with its maximal operator space structure, by a well known result
     (see Exercise 28.1 in \cite{P4}) we have also a lower bound $\E\|S\|\ge \alpha n$ provided $n\le N$. In this case, the   remaining estimate 
     of $\|a\|_{RC}$ required to complete the proof can be found in \cite[p. 223]{P4}.
        \end{proof}
  \section{Examples of  non exact  subexponential $C^*$-algebras}\label{s7}
  In this section, we will show that the   (random) $C^*$-algebra generated by the block direct sum of a sequence
  of i.i.d. random matrices is almost surely subexponential with constant 1 and not exact. We will first isolate,
  in Theorem \ref{t1} below, the
  properties of a {\it deterministic} (non random) sequence of   block direct sum operators that guarantee that the
  generated $C^*$-algebra has the desired properties. For that purpose, we need some preparation.
  
     Consider the direct sum $B=\oplus_{m\ge 1} M_{m}$. By definition, for any $x=\oplus_{m\ge 1} x(m)\in B$
we have $\|x\|=\sup_{m\ge 1}\| x(m)\|.$
   We equip $M_{m}$ with its normalized trace $\tau_m$. 
           
          Let $u_j=\oplus_m u_j(m)$ be elements of $B$. Let  $\cl A$ be the  unital $C^*$-algebra generated by $u_1,u_2,\cdots ,u_n$. For simplicity we set $u_0=1$.
          Let $\cl C$ be a  unital $C^*$-algebra that we assume generated by $c_1,c_2,\cdots $ 
          and equipped with a faithful
          tracial state $\tau$. We again set $c_0=1$.
          
          We say (following \cite{M}) that $\{u_j(m)\mid 1\le j\le n\}$ tends strongly to $\{c_j\mid 1\le j\le n\}$
          when $m\to \infty$ if  it tends weakly 
          (meaning ``in moments" relative to $\tau_m$ and $\tau$) and moreover
          $  \|P(u_i(m))\| \to  \|P(c_i)\|$ for any (non-commutative) $*$-polynomial $P$.
          This implies that
          for any finite set $P_0,P_1,\cdots,P_q$ of  such polynomials , for any $k$ and any $a_j\in M_k$ we have
          \begin{equation}\label{eqq0--} \lim_{m\to \infty} \|\sum\nolimits_0^q a_j \otimes P_j (u_i(m)) \|= \|\sum\nolimits_0^q a_j \otimes P_j(c_i) \|.\end{equation}
         In particular we have
          \begin{equation}\label{eqq0-} \lim_{m\to \infty} \|\sum\nolimits_0^n a_j \otimes u_j(m) \|= \|\sum\nolimits_0^n c_j \otimes a_j \|.\end{equation}
          Let $I_0\subset B$ denote the ideal of sequences $(x_m)
\in  B$ that tend to zero in norm (usually denoted by $c_0(\{M_{N_m}\})$.
          Let $Q:\ B\to B/I_0$ be the quotient map.
          It is easy to check that for any polynomial $P$ we have
          $\|Q(P(u_j))\|=\|P(c_j)\|$.
          So that, if we set $I=I_0\cap \cl A$,  we  have a natural identification
          $$\cl A/I= \cl C.$$

         Let $P_d$ denote the linear space of all polynomials of degree $\le d$
         in the non commutative variables $ (X_1,\cdots,X_n,X^*_1,\cdots,X^*_n)$.
         We will need to consider the space $M_k \otimes P_d$.  
         It will be convenient to systematically use the following notational convention:
         $$\forall 1\le j\le n\quad X_{n+j}=X_j^*.$$
         A typical element
         of $M_k \otimes P_d$ can then be viewed as a polynomial $P=\sum a_{J}\otimes X^J$ with coefficients
         in $M_k$. Here the index $J$ runs over the   disjoint union 
         of the sets $\{1,\cdots,2n\}^i$ with $1\le i\le d$. We also
         add symbolically  the value $J=0$ to the index set and we set $X^0$ equal to the unit.

         We denote by $P(u(m))\in M_k\otimes M_m$ (resp. $P(c)\in M_k\otimes \cl C$)
         the result of substituting $\{u_j(m)\}$ (resp. $\{c_j\}$)    in place of $\{X_j\}$.
         It follows from the strong convergence of
         $\{u_j\mid 1\le j\le n\}$  to $\{c_j\mid 1\le j\le n\}$ that
         for any $d$ and any $P\in M_k \otimes P_d$ we have
         $$\|P(u(m))\|\to \|P(c)\|.$$
         With a similar convention we will write e.g.
         $P(c)=\sum a_{J}\otimes c^J$.\\
  In particular this implies (actually this already follows from weak convergence)
          \begin{equation}\label{ee1} 
       \forall k\  \forall d \ \forall P\in  M_k \otimes P_d\quad   \|P(c)\|\le \liminf_{m\to \infty}\|P (u(m))  \|.
          \end{equation} 
                     \begin{rem}\label{r1} Let us  write $P$ as a sum of monomials $P=\sum a_{J}\otimes X^J$
         as above. We will assume that the  operators $\{c^J\}$ are linearly independent.
           From this assumption
          follows that there is a constant $c_2(n,d)$ such that
         $$\sum_{J} \|a_{J}\|\le c_2(n,d) \|P(c)\|.$$
         Indeed, since the span of the $c^J$'s is finite dimensional,  the linear form that takes $P$ to its $c^J$-coefficient  
         is continuous, and its norm (that depends obviously only on $(n,d)$)
         is the same as its c.b. norm. Of course this depends also on the distribution of the family  $\{c_j\}$ but we view this as fixed from now on.
         \end{rem}
         
    We will consider the following assumption: 
   
\begin{equation}\label{eea2}
 \sum_1^n \tau(|c_j|^2) >  \|   \sum_1^n u_j \otimes \bar c_j \|_{{\cl A}\otimes_{\min}  \bar {\cl C}} .     \end{equation} 
                
                   \n {\bf Notation.} Let $\alpha\subset \NN$ be a subset (usually infinite in the sequel).
                   We denote $$B(\alpha)= \oplus_{m\in \alpha} M_m.$$
                   $$u_j(\alpha)=\oplus_{m\in \alpha} u_j(m)\in B(\alpha).$$
                   We will denote by $A(\alpha)\subset B(\alpha)$ the unital  $C^*$-algebra   generated   by
            $\{ u_j(\alpha)\mid 1\le j\le n\}$. With this notation  $\cl A=A(\NN)$ and $u_j=u_j(\N)$. \\
            We also set 
             $E_d(\alpha) = \{P(u(\alpha))\mid P\in P_d\}.$
          \def\a{\alpha}

          Fix a degree $d\ge 1$. Then for any real numbers $m\ge 1$  and $t\ge 1$ we define
          $$C_d(t,m)=\sup_{m'\ge m}\sup_{k\le t}\{ \|P(u(m'))\|\mid P\in M_k\otimes P_d,\ \|P(c)\|\le 1\}.$$
            \begin{thm} \label{t1} Assume that for any $d\ge 1$ there are $a>0$ and $D>0$
            such that $C_d(N,aN^D)\to 1$ when $N \to \infty$. Assume moreover that
            \eqref{ee1} holds and that $\cl C$ is exact. Then for any subset $\alpha\subset \NN$ the unital  $C^*$-algebra $A(\alpha)$ generated by
            $\{ u_j(\alpha)\mid 1\le j\le n\}$ is $1+\vp$-subexponential for any $\vp>0$.
            Moreover, if we assume 
            that  that $\{u_j(m)\mid 1\le j\le n\}$ tends weakly 
            (meaning ``in moments" relative to $\tau_m$ and $\tau$) to $\{c_j\mid 1\le j\le n\}$
          when $m\to \infty$  
          and if            \eqref{eea2} holds then $A(\alpha)$ is not exact.
             \end{thm}

       \begin{proof}  For subexponentiality, we need to show that for any fixed $\vp>0$ and any finite dimensional subspace $E\subset A(\a)$ the growth of $N\mapsto K_E(N,1+\vp)$ is subexponential. Since the polynomials in $\{u_j(\a)\}$ are dense in $A(\alpha)$, by perturbation it suffices to check this
       for $E\subset E_d(\alpha) $. Thus we may as well assume $E=E_d(\alpha) $.\\
       Then we may choose $N_0$ large enough so that $C_d(N,aN^D)<1+\vp$ for all $N\ge N_0$.
       We claim that for all $N\ge N_0$ we have $K_E(N,1+\vp)\in O(N^{2D})$ when $N\to \infty$.
       To verify this, let $P\in M_N\otimes P_d$. 
       Then, recalling \eqref{ee1}, we have 
        \begin{equation}\label{ee2} \|P(c)\|\le \sup_{m\ge aN^D}
\|P(u(m))\|\le C_d(N,aN^D)\|P(c)\|.\end{equation}
Let $\a'= \a \cap [1,aN^D)$.
Let $T:\ E \to        B(\a')\oplus \cl C$ be the linear mapping
defined for all $P$ in $P_d$ by
$$T(P(u(\a))= P(u(\a'))\oplus P(c).$$
We may assume $\a$ infinite (otherwise the subexponentiality is trivial). Then \eqref{ee1} shows
that $\|T\|_{cb}\le 1$. Conversely, by \eqref{ee2} we have
$$\|(T^{-1})_N\|\le C_d(N,aN^D)<1+\vp.$$
Let $\hat E$ be the range of $T$.
 This shows that $d_N(E,\hat E)< 1+\vp.$ 
 We have $\hat E\subset \oplus_{k< aN^D} M_k \oplus \hat E'$
 where $\hat E'$ is a finite dimensional subspace of $\cl C$ (included in the span of polynomials of degree $d$). 
  Since $\cl C$ is exact, there is an integer $K$ such that, for any $\vp'>0$,  $\hat E'$
  is completely $(1+\vp')$-isomorphic to a subspace of  $ M_K$, so that
   $\hat E$ is completely $(1+\vp')$-isomorphic to a subspace of  $ \oplus_{k< aN^D} M_k \oplus M_K$.
  Therefore, if we choose $\vp'$ such that  $(1+\vp')d_N(E,\hat E)< 1+\vp$, we have for any $N\ge N_0$
  $$K_E(N,1+\vp)\le 1+2+\cdots+ [aN^D]+K \in O(N^{2D})$$
  and hence our claim follows, proving the $1+\vp$-subexponentiality.\\
  More precisely, taking the subset $\a $ into account 
  we wish to record here for future reference that
   \begin{equation}\label{e11}
   K_E(N,1+\vp )\le \sum_{k< aN^D,k\in \a} k +K. 
   \end{equation}
  We now show that $A(\alpha)$ is not exact. Recall the notation
  $B(\alpha)=\oplus_{m\in \a} M_m$. 
  By Kirchberg's results (see e.g. \cite[p. 286]{P4}), if $A(\alpha)$  is exact then
  the inclusion map $V:\ A(\alpha)\to B(\alpha)$ is approximable by
  a net of maps factoring   completely positively through matrix algebras. In particular, it satisfies the following:
  for any $C^*$-algebra $C$ the mapping
    $V\otimes Id_{C}:\ A(\alpha)\otimes_{\min} C \to B(\alpha)\otimes_{\max} C$ 
    is bounded (and is actually contractive). 
    Let $\cl U$ be any free ultrafilter on $\a$. 
    \def\u{{\cl U}}
    Let $M^\u\subset B(L_2(\tau_\u))$ denote the von Neumann algebra ultraproduct of $\{M_m\mid m\in \a\}$, with each 
    $M_m$ equipped with $\tau_m$. The Hilbert space $L_2(\tau_\u)$ is the one obtained by 
    the GNS construction applied to $B(\alpha)=\oplus_{m\in \a} M_m$ equipped with the state
    $\tau_\u=\lim_\u \tau_m$.
    Recall that $\tau_\u$ is a faithful normal tracial state on $M^\u$ (cf. e.g. \cite[p. 211]{P4}).
   By the  weak convergence assumption, 
     we may view $\cl C$ as embedded in $M^\u$, in such a way that
     $\tau_\u $ restricted to $\cl C$
coincides with $\tau$.     
    Let $\cl M$ be the von Neumann algebra generated by $\cl C$.
 We have
    a quotient map $Q_1:\ B(\alpha)\to M^\u$ and
    a (completely contractive) conditional expectation $Q_2$  from  $M^\u$ to $\cl M$.
    Let $q:  \ A(\alpha)\to \cl M$ be the composition $q=Q_2Q_1 V$. 
    By the above, $q\otimes Id_{C}:\ A(\alpha)\otimes_{\min} C \to \cl M\otimes_{\max} C$ 
   must be  bounded (and   actually contractive). However, if we take
   $C=\bar {\cl C} $, this implies since $c_j=q(u_j(\a))$
   $$   \|\sum_1^n c_j \otimes \bar c_j \|_{{\cl M}\otimes_{\max}  \bar {\cl C}}\le 
     \|   \sum_1^n u_j (\a)\otimes \bar c_j \|_{{A(\alpha)}\otimes_{\min}  \bar {\cl C}}\le
    \|   \sum_1^n u_j \otimes \bar c_j \|_{{\cl A}\otimes_{\min}  \bar {\cl C}} .$$
    But now  using the fact that left and right multiplication acting on $L_2(\tau_\u)$ are commuting representations on $\cl M$,
    we immediately find
    $$ \sum\nolimits_1^n \tau(|c_j|^2) =\sum\nolimits_1^n \tau_\u(|c_j|^2) \le   \|\sum_1^n c_j \otimes \bar c_j \|_{{\cl M}\otimes_{\max}  \bar {\cl C}} $$
    and this contradicts \eqref{eea2}. This contradiction shows that $A(\alpha)$
    is not exact.
 \end{proof}
   \begin{rem}\label{r02}  More generally, let $E$ 
              be any finite dimensional subspace spanned by  a finite set of  polynomials in the generators $\{u_j\}$.
              Let
              $$C_E(t,m)=\sup_{m'\ge m}\sup_{k\le t}\{ \|P(u(m'))\|\mid P\in M_k\otimes E,\ \|P(c)\|\le 1\}.$$  
          Let $E(\a)$ denote the subspace of ${A(\alpha)}$ formed of the corresponding polynomials
          in      the generators $\{u_j(\a)\}$ of ${A(\alpha)}$.       The preceding proof shows 
            more precisely that if we assume that there is a constant $C$ such that
            $C_E(N,aN^D)\le C$ for all $N$ large enough then
            for any $c>C$ we have for any $\a\subset \NN$
            $$ K_{E(\a)}(N,c)\in O\left(\sum\nolimits_{k< aN^D,k\in \a} k\right).$$  
        \end{rem}

              \begin{rem}\label{r2}  
            Let  $Y^{(m)}$ denote a random $m\times m$-matrix with i.i.d. 
  complex Gaussian entries with mean zero and $L_2$-norm equal to $m^{-1/2}$, and let
     $(Y_j^{(m)})$ be a sequence of i.i.d. copies of $Y^{(m)}$.
     We will use the matrix model formed by these matrices
              (sometimes called the ``Ginibre ensemble"), for which it
            is known (\cite{VDN}) that we have weak convergence to a free circular family $\{c_j\}$.
            Moreover, by \cite{HT3} we have also almost surely strong convergence
            of the random matrices to the free circular system. Actually, the inequalities
            from \cite{HT3,HT4} that we will crucially use are stated there mostly for the GUE ensemble, i.e. for 
            self-adjoint Gaussian matrices with a semi-circular weak limit, and for self-adjoint polynomials in them
            with matrix coefficients.
            These can be defined simply by setting
            $$X_j^{(m)}=\sqrt {2 }\Re(Y_j^{(m)}).$$
            Note  we also have an identity  in distribution $s_j= \sqrt {2 } \Re ( c_j)$.
            We call this the self-adjoint model.
            However, as explained in \cite{HT3} , it is easy to pass from the self-adjoint case to the general one by
            a simple ``$2 \times 2$-matrix trick".
            Since we prefer to work in the circular setting,   with polynomials
            in $c_j,c_j^*$ (we call those $*$-polynomials) we will now indicate this trick.
        
          When working in the self-adjoint model,  of course we consider only polynomials
            of degree $d$ in $ (X_1,\cdots,X_n)$. Fix $k$. Then
            the set  of polynomials of degree $\le d$ with coefficients in $M_k$ of the form $P(X_j^{(m)})$
            is included in the corresponding set of $*$-polynomials of degree $\le d$
            of the form $P(Y_j^{(m)})$. Conversely, any $P(Y_j^{(m)})$ can be viewed
            as a polynomial of degree $\le d$ in $ (X_1^{(m)},\cdots,X_{2n}^{(m)})$. Indeed, the real and imaginary
            parts of  $Y_j^{(m)}$ are independent copies of $X_j^{(m)}$. Thus, by
            this simple argument, we can replace   $*$-polynomials (with coefficients in  $ M_k$)      in  $(Y_j^{(m)})$  
            by   polynomials (with coefficients in  $ M_k$) in $(X_j^{(m)})$. 
            However, there is a further restriction: The results
            of \cite{HT3}  are stated only for {\it self-adjoint} polynomials with coefficients in  $ M_k$.
           But then 
             the trick (indicated in \cite{HT3}) to deal with this consists in replacing a  general polynomial   $P\in M_k\otimes P_d$
            with coefficient in $M_k$
                     by a self-adjoint one $\hat P$ with coefficient in $M_{2k}$ defined by
         $$\hat P =\left(\begin{matrix} 0 \ \ P\\
            P^* \ \  0\end{matrix} \right)\in M_{2k}\otimes P_d.
         $$
         One then notes that $\|\hat P(s)\|=
         \|  P(s)\|$ and similarly $\|\hat P(X_j^{(m)})\|=
         \|  P(X_j^{(m)})\|$. Thus, for instance, by simply passing from $k$ to $2k$
         we can deduce the strong convergence for arbitrary polynomials,
         as expressed in  \eqref{eqq0--} and \eqref{eqq0-} from   
         the   self-adjoint case. 
            
               \end{rem}           
           The following Lemma is well known.
             
        \begin{lem} Let $F$ be any scalar valued random variable that is in $L_p$ for all $p<\infty$.
        Fix ${\theta}>0$.
        Assume that 
        $$\sup_{p \ge 1}  p^{-{\theta}}\|F\|_p \le \sigma.$$
        Then
        $$\forall t>0 \quad \P\{ |F|>t\} \le e \exp - (e\sigma)^{-1/{\theta}} t^{1/{\theta}} . $$
       
               \end{lem}
         \begin{proof} By Tchebyshev's inequality, for any $t>0$ we have
         $t^p  \P\{ |F|>t\}\le  (\sigma p^{\theta})^p $, and hence
         $ \P\{ |F|>t\} \le  (t^{-1}\sigma {p^{\theta}})^p \le \exp -p\log (t /(\sigma {p^{\theta}}))$.
         Assuming $ t/(e\sigma )\ge 1$, we can choose $p= (  t/(e\sigma ) )^{1/{\theta}}$ and then we find
         $\P\{ |F|>t\} \le   \exp -(e \sigma)^{-1/{\theta}} t^{1/{\theta}}  $ and, a fortiori, the inequality holds. Now if  $ t/(e\sigma )< 1$, we have
         $\exp {-(e\sigma)^{-1/{\theta}} t^{1/{\theta}} }> e^{-1} $ and hence $e \exp{- (e\sigma)^{-1/{\theta}} t^{1/{\theta}} }>1$ so that the inequality 
         trivially holds.
   \end{proof}
   We will use concentration of measure in the following form:
           
             \begin{lem}\label{lem3} There is a constant $c_1(n,d)>0$   such that               for any $k$ and any
                $P\in M_k\otimes P_d$ with $\|P(c)\|\le 1$, we have
                $$\forall t>0\quad \P\{ |\|P(Y^{(m)})\|-\E\|P(Y^{(m)})\|| > t \} \le e\exp-(t^{2/d} m^{1/d}/ c_1(n,d)).$$
               
               \end{lem}
               \def\b{\beta}
         \begin{proof}    This  follows from
         a very general concentration inequality for Gaussian random vectors, that can be derived in various ways.  We choose the following for which we refer to \cite{P02}.
         Consider any sufficiently smooth function
         (meaning a.e. differentiable) $f:\ \R^n \to \R$ and let $\P$ denote the canonical Gaussian measure on
         $\R^n$. Assuming $f\in L_p(\P)$ we have 
         $$ \|f-\E f\|_p\le (\pi/2)\| Df(x).y\|_{L_p(\P(dx)\P(dy))}. $$
         Let $\gamma(p)$ denote the $L_p$-norm of a standard normal Gaussian variable
         (in particular $\gamma(p)=\|f\|_p $ for $f(x)=x_1$). 
         Recall that $\gamma(p)\in O(\sqrt{p} )$ when $p\to \infty$. Thus the last inequality implies that there is a constant $\b $ such that
         $$ \|f-\E f\|_p\le \b \sqrt{p} \| \|Df(x)\|_2 \|_{L_p(\P(dx))}, $$
         where $\|Df(x)\|_2$ denotes the Euclidean norm of the gradient of $f$ at $x$. 
         Clearly this remains true  for any $f$ on $\C^n$ (with the gradient computed on $\R^{2n}$).
         
         We will apply this
         to a function $f$ defined on $(\C^{m^2})^n$.
         We need to first clarify the notation. We identify $\C^{m^2}$ with $M_m$. Let $P\in M_k \otimes P_d$. Then we 
          define    $f$  on    $(\C^{m^2})^n$  by
         $$f(w_1,\cdots,w_n)   =\| g(w_1,\cdots,w_n)   \|$$
         with 
          $$g(w_1,\cdots,w_n)   =P(m^{-1/2}w_1,\cdots,m^{-1/2}w_n, m^{-1/2}w^*_1,\cdots,m^{-1/2}w^*_n ).$$ 
         Note that  for this choice of $f$ the derivative $D_z$ in any direction $z$ satisfies
         $D_z f \le \| D_zg\|$ and hence taking the sup over $z$ in the  Euclidean  unit sphere, we have pointwise
         $$\|Df\|_2 \le \sup\nolimits_z \|D_zg\|.$$
     In order to majorize $\sup\nolimits_z \|D_zg\|$, we first invoke Remark \ref{r1}.
Using the bound in that remark,    we are 
         reduced  to majorize in  the case when $P(   X )=X^J$,  a product of $\ell$ terms, with $\ell\le d$.\\ Then, we claim that  $D_z g$ is the sum
         of  $\ell$ terms of the form  $m^{-1/2} a z_i b$ satisfying, for all $z=(z_i)$ in the Euclidean sphere,
       the bound  $$\|m^{-1/2} a z_i b\|\le m^{-1/2}\|a\|\|b\|\le m^{-1/2} \sup\{ \|m^{-1/2}w_j\|\mid 1\le j\le n\}^{\ell-1} .$$
         Indeed, if  $P(   X )=  X^J= X_{j_1}\cdots X_{j_\ell}$ ($1\le \ell\le d$), 
         and if $g$ is associated as above (with say $a_J=I$), then $g$
         is of the form         $g=y_{j_1}\cdots y_{j_\ell}$ with $y_j=m^{-1/2}w_j, y_{n+j}=m^{-1/2}w^*_j$,
         and hence
         $D_z g=\sum_i  y_{j_1}\cdots (D_z y_{j_i})\cdots y_{j_\ell}$
         and    $D_z y_{j_i}$ is equal to  $m^{-1/2} z_{j_i}$.
         Note $\|z_i\|\le \|z_i\|_2$ and hence $\|a z_i b\|\le \|a\|\|b\|$.  From this the claim follows.
         
                 Recollecting all the terms , this yields
        a pointwise estimate at the point $w \in M_m^{n}$
        $$\sup_z \|D_z g\| \le c_3(n,d) m^{-1/2} \sup\{ \|m^{-1/2}w_j\|^{\ell-1}\mid 1\le j\le n, 1\le \ell\le d \}.$$
        Thus we obtain
      $$ \|f-\E f\|_p\le \b \sqrt{p} c_3(n,d) m^{-1/2} \| \sup_{1\le j\le n,\  \ell\le d}  \|Y^{(m)}_j\| ^{ \ell-1}\|_p,$$
      and a fortiori 
         \begin{equation}\label{e100} \|f-\E f\|_p\le \b\sqrt{p}  c_3(n,d) m^{-1/2}  \sum\nolimits_{1\le j\le n,\  \ell\le d} \|  \|Y^{(m)}_j\| ^{\ell-1}\|_p.   \end{equation} 
         Now by general results on integrability
         of Gaussian vectors (see \cite[p. 134]{Led}),
         we know that
         there is an absolute  constant $c_5$ such that 
         $$\|\|Y^{(m)}_1\| ^{\ell-1}\|_p= \|Y^{(m)}_1\|^{\ell-1}_{L_{p(\ell -1)}(M_m)}       
            \le (c_5\sqrt{ p(\ell -1)} \E \|Y^{(m)}_1\| )^{\ell -1} $$
            and since we know that $\E \|Y^{(m)}_1\| \to 2$ when $m\to \infty$  
            it follows that $\|\|Y^{(m)}_1\| ^{\ell -1}\|_p\le (c_{6} \sqrt{ p(\ell -1)}  )^{\ell -1} 
            \le (c_{6} \sqrt{ p(d -1)}  )^{d -1}=(c_{6} \sqrt{  (d -1)}  )^{d -1} p^{d/2-1/2}$ for some numerical constant $c_{6}>1$.
            Thus, by \eqref{e100}  we obtain
            $$ \|f-\E f\|_p\le c_4(n,d) m^{-1/2} p^{d/2},$$
         and the conclusion follows from the preceding Lemma with ${\theta}=d/2$ and $\sigma=c_4(n,d) m^{-1/2}$.
              \end{proof}

                 \begin{rem}\label{r3}  It will be convenient to record here an elementary consequence of
                 Lemma \ref{lem3}.
                 Let $F= \|P(Y^{(m)})\|$ and let  $t_m=  \E\| P(Y^{(m)})\|$, so that 
                 we know
                 $\forall t>0\quad \P\{ F > t+t_m \} \le \psi_m(t)$
                 with $$\psi_m(t)=
                  e\exp-(t^{2/d} m^{1/d}/ c_1(n,d)).$$
                 We have
                $$ \E\left((F/2-t_m) 1_{\{F/2>t_m\}}\right)=\int_{t_m}^\infty \P\{F/2>t\} dt\le  \int_{t_m}^\infty \P\{F>t+t_m\} dt\le \int_{t_m}^\infty  \psi_m(t)dt $$
                and hence
                \begin{equation}\label{e10} \E F 1_{\{F/2>t_m\}} \le 2t_m \P{\{F/2>t_m\}}+  2\int_{t_m}^\infty  \psi_m(t)dt  .\end{equation}

                          \end{rem}
              
           The next result is a consequence of the results of {Haagerup} and {Thorbj{\o}rnsen}
            \cite{HT3} and of them with Schultz \cite{HT4}. Let us first recall
        the result from \cite{HT3} that we crucially need.
        
        \def\CB{{\mathcal B}}
 
   \begin{thm}[\cite{HT3,HT4}]\label{hst} Let $\chi_d(k,m) $ denote the best constant 
            such that
            for any $P\in M_k\otimes P_d$
            we have
            $$\E \|P( Y_j^{(m)}) \|\le \chi_d(k,m) \| P(c) \|. $$
            Then for any $0<\delta<1/4 $
            $$\lim_{m\to \infty }  \chi_d([m^{\delta}],m)=1.$$
               \end{thm}
         \begin{proof} Let $\chi'_d(k,m) $ be defined exactly
         as $\chi_d(k,m) $ but in the self-adjoint setting, i.e.
          with $\{X_j^{(m)} \}$ in place of $\{Y_j^{(m)} \}$ and a free semicircular system
       $\{s_j \}$ (and $P(s)$)  in place of $\{c_j \}$ (and $P(c)$). By Remark \ref{r2} it suffices to prove
          that $\lim_{m\to \infty }  \chi'_d([m^{\delta}],m)=1.$\\
      We will now majorize $ \chi'_d([m^{\delta}],m)$. By homogeneity we may assume $\| P(s) \|=1$.
        Then by Remark \ref{r1} we also have
        \begin{equation}\label{e12}\sum_{J} \|a_{J}\|\le c_2(n,d).\end{equation}
      Fix $\vp>0$ and $t>1+\vp$.
Consider a    function        $\varphi\in
C_c^\infty(\RR,\RR)$ with values in $[0,1]$
such that $\varphi =0 $ on $[-1,1]$ 
and $\varphi (x)=1 $ for all $x$ such that $1+\vp<|x|<t$
and $\varphi (x)=0$ for  $|x|>2t$.
 Let $P^{(m)}=P( X_j^{(m)})$ and $P^{(\infty)}=P( s_j)$.
 By Remark \ref{r2} we can reduce our estimate
 to the case of a self-adjoint polynomial in $(X_j^{(m)})$.
  Then  by  \cite{HT4} 
  (and by very carefully tracking the dependence of the various constants in \cite{HT4})  we have for $m\ge c_{13}(n,d)$
\begin{equation}\label{hot}
\E\big\{(\tau_k\otimes\tau_m)\varphi(P^{(m)})\big\} =
(\tau_k\otimes\tau)\varphi(P^{(\infty}))+R_m(\varphi)
\end{equation}
where
\begin{equation}
\label{eq5-12}
|R_m(\varphi)| \le   k^3 m^{-2} c_9(n,d) c_\vp t^3
\end{equation}
where $c_\vp$ depends only on $\vp$.
 Note $\varphi(P^{(\infty)})=0$.  Therefore
 \begin{equation}\label{htt}
\E\big\{(\tau_k\otimes\tau_m)\varphi(P^{(m)})\big\} \le k^3 m^{-2} c_9(n,d) c_\vp t^3.
\end{equation}
Since $\|P^{m}\|\in (1+\vp,t) \Rightarrow (\tau_k\otimes\tau_m)\varphi(P^{m})\ge 1/(km)$
  by Tchebyshev's inequality we find
$$ \P\{ \|P^{(m)}\|\in (1+\vp,t) \} \le  (km)k^3 m^{-2} c_9(n,d) c_\vp t^3=k^4m^{-1} c_9(n,d) c_\vp t^3.$$
        Thus we obtain
        $$\E \|P^{(m)}\|\le 1+\vp + k^4 m^{-1} c_9(n,d) c_\vp t^4+ \E (\|P^{(m)}\| 1_{\{\|P^{(m)}\|> t\}}).  $$
We will now invoke \eqref{e10}:  choosing $t=2t_m=2\E \|P^{(m)}\|$ we find
 $$\E \|P^{(m)}\|\le 1+\vp + k^4 m^{-1} c_9(n,d) c_\vp t_m^4+   2t_m \psi_m(t_m) +  2\int_{t_m}^\infty  \psi_m(t) .  $$
 Now by \eqref{e12} and by H\"older
 we  have 
 $$t_m\le c_2(n,d) \sup_J \E\| {X^{(m)}}^J \|\le c_2(n,d) \sup_{|J|\le d} \E(\|X_1^{(m)}\|^{|J|}) $$
 but  by a well known result essentially due to Geman \cite{G} (cf. e.g. \cite[Lemma 6.4]{S}),
 for any $d$ we have
 $$c_9(d)=\sup_m  \E(\|X_1^{(m)}\|^{d})<\infty.$$
 Therefore we have  $t_m\le c'_2(n,d)$.
 We may assume $t_m>1$ (otherwise there is nothing to prove) and
 hence  we have proved 
 $$\E \|P^{(m)}\|\le 1+\vp + k^4m^{-1} c'_9(n,d) c_\vp  +   2c'_2(n,d) \psi_m(1) +  2\int_{1}^\infty  \psi_m(t)dt . $$
Thus for any $\vp>0$ we conclude 
 \begin{equation}\label{htt+}\chi'_d(k,m)\le 1+\vp + k^4m^{-1} c'_9(n,d) c_\vp  +   2c'_2(n,d) \psi_m(1) +  2\int_{1}^\infty  \psi_m(t)dt.\end{equation}
 From this estimate
 it follows clearly that for any $0<\delta<1/4 $
 $$\limsup_{m\to \infty}  \chi'_d([m^{\delta}],m)\le 1+\vp  .$$
  
 \end{proof}

            \begin{lem}\label{lem4} Fix integers $d,k,m$. Let $\chi_d(k,m) $ denote the best constant 
            appearing in Theorem \ref{hst}.
            Then for any $\vp>0$ there are positive constants $c_7(n,d,\vp)$ and $c_8(n,d,\vp)$ such that
            if $k$ is the largest integer such that  $m\ge c_7(n,d,\vp) k^{2d}$ 
            the set
  $$\Omega_{d,\vp}(m)=
   \{   \forall P\in M_k\otimes P_d\quad   \|P(Y^{(m)}(\omega)) \|\le  (1+\vp) (\chi_d(k,m)+\vp) \| P(c) \|  \}$$
            satisfies
            $$\P(\Omega_{d,\vp}(m)^c)\le e\exp\left(  -m^{1/d}/ c_8(n,d,\vp)  \right).$$
            
       \end{lem}
         \begin{proof} 
         For any $P\in M_k\otimes P_d$ with $\|P(c)\|\le 1$,
         we have by Lemma
          \ref{lem3}  for any $t>0$
          $$\P\{ \|P(Y^{(m)}) \|>t+\chi_d(k,m)\}\le   e\exp-(t^{2/d} m^{1/d}/ c_1(n,d)).$$
          Let $\cl N$ be a $\delta$-net in the unit ball
          of the space $P_d$ equipped with the norm $P\mapsto \|P(c)\|$.
          Since $\dim(M_k\otimes P_d)=c_6(n,d) k^2$ for some $c_6(n,d)$, it is known that we can find such a net with
          $$|\cl N|\le (1+2/\delta)^{c_6(n,d) k^2}.$$
          Let $\Omega_1= \{\forall a\in \cl N,  \|P(Y^{(m)}) \|>t+\chi_d(k,m)\}$.
          Clearly 
          $$\P(\Omega_1)\le |\cl N|  e\exp-(t^{2/d} m^{1/d}/ c_1(n,d))
           \le  e\exp\left (2 c_6(n,d)\delta^{-1}    k^2 -t^{2/d} m^{1/d}/ c_1(n,d)  \right) . $$
Thus if we choose $m$ so that (roughly )
$t^{2/d} m^{1/d}/ c_1(n,d)=4 c_6(n,d)\delta^{-1}    k^2$ we find an estimate of the form
$$\P(\Omega_1)\le  e\exp\left (  -t^{2/d} m^{1/d}/ 2c_1(n,d)  \right) . $$
Note that on the complement of $\Omega_1$ we have
$$\forall P\in \cl N \quad  \|P(Y^{(m)}) \|\le t+\chi_d(k,m) .$$
By a well known result  (see e.g. \cite[p. 49-50]{P-v}) we can pass from the set $\cl N$ to the whole unit ball at the cost
of a factor close to 1, namely we have on the complement of $\Omega_1$
$$\forall P\in M_k\otimes P_d\quad   \|P(Y^{(m)}) \|\le (1-\delta)^{-1} ( t+\chi_d(k,m) )\|P(c)\|.$$
Thus if we set $t= \vp$ and $\delta\approx \vp/2$, we obtain that
if $m\ge c_7(n,d,\vp) k^{2d}$
we have a set $\Omega_1'=\Omega_1^c$
with 
$$\P({\Omega_1'}^c)\le   e\exp\left (  -\vp^{2/d} m^{1/d}/ 2c_1(n,d)  \right) ,$$
 such that for any $\omega\in\Omega_1'$  we have 
$$\forall P\in M_k\otimes P_d\quad   \|P(Y^{(m)}(\omega)) \|\le (1+\vp) (\chi_d(k,m)+\vp)\|P(c)\|.$$
        \end{proof}

         \begin{thm}\label{t2} For any infinite subset $\a\subset \NN$ and  each $j$ let  $u_j(\a)(\omega)$ denote the block direct sum defined by
$$u_j(\a)(\omega)= \oplus_{m\in \a} Y_j^{(m)}(\omega)\in  \oplus_{m\in \a} M_m,$$
and let $u_j(\omega)=u_j(\N)(\omega)$.
Let  $A(\a)(\omega)$  denote the $C^*$-algebra   generated by the infinite sequence
$\{u_j(\a)(\omega) \mid j=1,2,\cdots\}$.  Then,  for almost every $\omega$,   the $C^*$-algebras $A(\a)(\omega)$ are  all subexponential with constant 1 but  are not exact.\\
Moreover, these results remain valid
in the self-adjoint setting, if we replace $u_j(\a)(\omega)$ by
 $$\hat u_j(\a)(\omega)= \oplus_{m\in \a} X_j^{(m)}(\omega)\in  \oplus_{m\in \a} M_m.$$
 \end{thm}
         \begin{proof} 
         By Lemma \ref{lem4} for any degree        
         $d$ and $\vp>0$ we have 
         $$\sum_m \P(\Omega_{d,\vp}(m)^c)<\infty. $$
         Therefore the set $V_{d,\vp}=\liminf_{m\to \infty} \Omega_{d,\vp}(m)$
         has probability 1. Furthermore  (since we may use a sequence of $\vp$'s tending to zero) we have
         $$\P( \cap_{d,\vp} V_{d,\vp})=1.$$
         Now if we choose $\omega$ in $\cap_{d\ge 1,\vp>0} V_{d,\vp}$, by Theorem \ref{hst}, the operators
        $u_j(\omega)$ satisfy the assumptions 
        of Theorem \ref{t1}, and hence $A(\a)(\omega)$ is $1$-subexponential for any $\a$.\\
        Recall that, by concentration
       (see Remark \ref{comp} below)
        $$ \sup\nolimits_{j\ge 1} \E (\|u_j\|^2) =\sup\nolimits_{j\ge 1}\E (\sup_m \|u_j(m)\|^2)<\infty.$$
        Therefore, by Fatou's lemma $$\E \liminf_{n\to \infty} n^{-1} \sum\nolimits_1^n \|u_j\|^2\le  \liminf_{n\to \infty} \E n^{-1} \sum_1^n \|u_j\|^2<\infty$$
         and hence
     there is a measurable set $\Omega_0\subset \Omega$ with $\P(\Omega_0)=1$
        such that
        $$\forall \omega\in \Omega_0\quad \liminf_{n\to \infty} n^{-1} \sum_1^n \|u_j(\omega)\|^2 <\infty.$$ 
        Therefore if we choose $\omega$ in the intersection
        of $ \cap_{d,\vp} V_{d,\vp}\cap \Omega_0$
        (which has probability 1)
        we find almost surely by \eqref{s2}
         $$\|   \sum\nolimits_1^n u_j (\omega)\otimes \bar c_j \|_{{\cl A}\otimes_{\min}  \bar {\cl C}}\le
        2\max\{ \|\sum u_j u_j^* \|^{1/2}, \|\sum u_j^* u_j \|^{1/2}\}
        \le   2( \sum\nolimits_1^n \|u_j(\omega)\|^2)^{1/2} \in O   (\sqrt{n} )$$
         so that \eqref{eea2} is satisfied when $n$ is large enough and hence $A(\a)(\omega)$ is not exact.\\
         Lastly, since $\{\hat u_j(\a)(\omega)\mid j\in \a\}$ has  the same   distribution as $\{\sqrt{2}\Re u_j(\a)(\omega)\mid j\in \a\}$ the
         random $C^*$-algebra they generate has ``the same distribution" as  $A(\a)(\omega)$, whence the last assertion. 
                \end{proof}

               \begin{rem}\label{comp} 
      In the preceding proof, we use the fact that $ \sup\nolimits_{j\ge 1} \E (\|u_j\|^2)
               <\infty.$ This is immediate if we assume that the vector valued random variables
               $\{u_j\mid j\ge 1\}$ are (stochastically) independent, since then they have automatically the same distribution
               so this reduces to $ \E (\|u_1\|^2)<\infty$. However, we claim this remains valid
               assuming merely, as we do, that for each $m$ the sequence $\{u_j(m)\mid j\ge 1\}$
               is independent. This follows rather easily from the common concentration
               of the variables $\{u_j\mid j\ge 1\}$. Indeed,   by a well known (rather soft) bound we have
               $$\Delta =\sup_{j,m\ge 1} \E\|u_j(m)\|<\infty.$$
               Then by \eqref{bor2} we have for any $j\ge1$
               $$\forall t>0\quad \P\{ \|u_j(m)\|>t+\Delta\}\le 2 \exp { -t^2 m}$$
               and hence for all $t>0$
               $$\P\{ \|u_j\|>t+\Delta\}\le \sum_m  \P\{ \|u_j(m)\|>t+\Delta\}\le  2\sum_m   \exp { -t^2 m}\le 2  (\exp { -t^2 })(1- \exp { -t^2 })^{-1},$$
               from which our claim is immediate.
                 \end{rem}
                 \begin{rem}
                  It seems clear that our results remain valid if we replace $(Y_j^{(m)})$ 
            by an i.i.d. sequence $(V_j^{(m)})$ of uniformly distributed $m\times m$ unitary matrices,
            and we replace $u_j$ by
            $$v_j(\a)(\omega)=\oplus_{m\in \a} V_j^{(m)}(\omega)\subset \oplus_{m\in \a}  M_m.$$
            Note that Collins and Male \cite{CM} proved that strong convergence holds in this case.            
            But,
            at the time of this writing,
            except for a partial result  in  Theorem \ref{cm} below,
            we have not   been able to prove that  $(V_j^{(m)})$ satisfies  the same estimates
             (e.g. as in Theorem \ref{hst}) as $(Y_j^{(m)})$.  \end{rem}
     However,  as this paper was being completed, Mikael de la Salle kindly communicated to me
     the proof of the following result.
                 \begin{thm}[de la Salle]
                 Let $\a\subset \N$. Let $A^{U}(\a)(\omega)$ be the $C^*$-algebra generated by the unitary operators
                 $\{v_j(\a)(\omega)\mid 1\le j\le n\}$. Then $A^{U}(\a)(\omega)$  is subexponential with constant 1.
                         \end{thm}
                         \begin{proof}[Sketch of proof] 
 Let $c$ be a circular variable, so that $s=\sqrt{2}\Re(c)$ is semi-circular.
                         Let $f:\ \R\to \T$ be a continuous function such that
                         $f(s)$ is uniformly distributed over $\T$. By
                         \cite{CM} we may assume $(Y_j^{(m)})$ and $(V_j^{(m)})$ defined
                         on a suitably enlarged probability space so that
                         $\|f(\sqrt{2}Y_j^{(m)})-V_j^{(m)}\|\to 0$ almost surely when $m\to \infty$.
                         Then by \cite{CM} the family $(f(\sqrt{2}Y_j^{(m)}))$ converges strongly to $(f(s_j))$
                         and the latter  is a family of free
                         Haar unitaries. 
                         Let $C^f_d(N,aN^D)$ be the analogue of $C_d(N,aN^D)$ 
                         but computed with $Y_j^{(m)}$ replaced by $f(\sqrt{2}Y_j^{(m)})$ 
                         and by $(c_j)$  replaced by $(f(s_j))$.
                                                  By the proof of Theorem \ref{t1} it suffices to show the following claim:
                         For any $d$ and $\vp>0$ there is $a,D$ such that
                         $\limsup_{N\to \infty} C^f_d(N,aN^D)\le 1+\vp$.
                          Assume  for a moment (although this is clearly wrong)  that $f$
                         is a polynomial of degree $q$. Then since a polynomial of degree $d$ in $(f(s_j))$
                         is also one of degree $\le qd$ in $(c_j)$, we have
                          $C^f_{qd}(N,aN^D)\le C_d(N,aN^D)$, and the claim follows.
                          To prove the claim  when $f$ is not a polynomial, one approximates $f$ 
                          uniformly on (say) the interval  $[-4,4]$ (which contains $[-2,2]$ in its interior)  by a polynomial $P$. Since
                          $\|\sqrt{2}Y_j^{(m)}\|\to 2$,  when $m$ is large, we can   approximate $f(\sqrt{2}Y_j^{(m)})$
                          by $P(\sqrt{2}Y_j^{(m)})$.
                     \end{proof}
          \begin{rem} By a unitary variant of the   argument  for non-exactness in Theorem \ref{t2},
          it is easy to see that,  whenever $\a$ is infinite and $n>2$,  $A^{U}(\a)(\omega)$ is 
          almost surely not exact (indeed note that if $(z_j)$ are free Haar unitaries we have
          $\|\sum_1^n v_j(\a)(\omega) \otimes \bar z_j\|=2\sqrt{n-1}>n$ for any $\omega$).
           \end{rem}
 \section{More examples  of non-exact subexponential    operator spaces  }\label{s6}
 In this \S  we   modify the preceding example
      to produce
        $n$-dimensional  operator spaces  $E$  with large exactness constant such that
        their associated sequence $K_E(N,C)$ grows
        as slowly as possible. We will obtain a growth of order  $O(N^2) $ when $N\to \infty$.
       So far this is the slowest growth we could produce among spaces with large exactness constant.
        
        By 
        Theorem \ref{hst} and Lemma \ref{lemcon}, we can apply Remark   \ref{r02} to the linear   
        span of  the generators
        with $a=1$ and $D$ equal to any number $>4$. This shows
          that if we consider $E=E_1(\a)$
        then  for  any $\vp ,\delta>0$ and any $\a$ and we have $K_E(N,1+\vp)\in O(N^{5+\delta})$ when $N \to \infty$.
       However, when $\a$ is a lacunary sequence (or when $1+\vp$ is replaced by $2+\vp$), we will show that  this estimate can be improved. 
        
                   \begin{lem}\label{blabla}      Let $E(\omega)\subset A(\omega)$
                   be the linear  span of $\{u_1(\omega),\cdots ,u_n(\omega)\}$
                   in $A(\omega)$.
                    With the notation in Remark \ref{r02}, for any $\vp>0$ there is a constant $c_\vp\ge 1$ (depending only on $\vp$) and a (measurable) subset $\Omega_2\subset \Omega$ with $P(\Omega_2)=1$ such that 
                   for  all $\omega\in \Omega_2$ we have  for all $N$ large enough (i.e. $\forall N\ge N_0(\vp,\omega) $)
                    $$C_{E(\omega)}(N,c_\vp n N^2) \le 2+\vp .$$
                    
              \end{lem}
                 \begin{proof}                To simplify the notation, let $k_m= [(m /(c_\vp n))^{1/2} ]$.
                  By
                 Lemma \ref{lemcon}    for each $m$
                 there is $\Omega(m)\subset \Omega$
                 with $\P  ( \Omega(m))> 1-3^{-2n{k_m}^2}$   such that
                 for any $\omega \in  \Omega(m)$  and any $k\le k_m$  
             (equivalently    $  c_\vp nk^2 \le m $)  
                  we have 
             \begin{equation}\label{es44}\forall (a_j) \in M_k^n\quad \| \sum a_j \otimes u^{(m)}_j(\omega)\| \le (2+\vp)\max\{\| (\sum a_j^*a_j)^{1/2} \|,\| (\sum a_ja_j^*)^{1/2}\| \}  .\end{equation}
  and a fortiori by \eqref{s2}
                  \begin{equation}\label{s1}\forall P \in M_k\otimes E \quad \|P(\{u_i^{(m)}\}\| \le (2+\vp)\|P(c)\|  .\end{equation}
 Since $\sum\nolimits_m  3^{-2n{k_m}^2}<\infty$,  it follows that 
 for almost  all $\omega\in \Omega$  \eqref{es44} must hold for all $m$ large enough,
 and a fortiori also \eqref{s1}.
                     This implies that 
                 $ C_{E(\omega)} (N,[c_\vp n N^2]) \le 2+\vp    $ for all $N$ large enough.                                 \end{proof}
        \begin{rem}   It should be possible to use 
        Collins and Male's results \cite{CM}   to replace Gaussian random matrices by unitary ones
        (uniformly distributed according to Haar measure), but we   could not  check this.
        For this   the Gromov-L\'evy isoperimetric inequality 
         (see \cite[\S 1.2 and \S 3.4]{GM}) should be used on $U(N)^n$ 
        instead of the Gaussian concentration of measure.
        \end{rem}

        Fix $0<\vp< 1$. We may  replace $c_\vp$ by a larger number, so we will assume for simplicity that $c_\vp$ is an integer,
        and we denote $a= c_\vp n $.
         Let us choose $N(m)$ inductively such that $N(0)=1$ and for any $m> 0$
       $$ N(m+1) =  c_\vp n N(m)^2 =aN(m)^2. $$ Thus
           \begin{equation}\label{eq25} \forall m\ge 0 \quad  N(m)= a^{2^m -1}. \end{equation}

                 \begin{lem}\label{ors}    Let $\a_2=\{N(m)\mid m\ge 0\}$. Then for  all
                 $\omega\in \Omega_2$, the space $E(\a_2)(\omega)$ satisfies
                 $K_{E(\a_2)(\omega)}(N,2+\vp)\in O(N^2)$ when $N\to \infty$. Moreover, this also holds
                 for ${E(\a')(\omega)}$ for any subset $\a'\subset \a_2$.  
                  \end{lem}
        \begin{proof}  By Remark \ref{r02} and Lemma \ref{blabla} we have      for all $N$ large enough 
        $$K_{E(\a)(\omega)}(N,2+\vp) \le \sum\nolimits_{m\ge 0, \ N(m)< aN^2} N(m),$$
        and $N(m)< aN^2$ iff $N(m-1)<N$, so that  $$\sum\nolimits_{m\ge 0, \ N(m)< aN^2} N(m)=\sum_{m\ge 0,\ N(m-1)<N} a^{2^m -1}
        =N(0)+N(1)+\cdots+N(q)$$
        where $q$ is   so that $N(q-1)<N\le N(q) $. Now there is clearly a constant $\gamma$ such that \\
        $N(0)+N(1)+\cdots+N(q)\le \gamma a^{2^q -1}\le \gamma a N(q-1)^2 < \gamma a N ^2 $, so that we find
        $$K_{E(\a)(\omega)}(N,2+\vp) < \gamma a N ^2.$$
   \end{proof}
   
   By  \eqref{es3} and \eqref{s2},  we know there is a (measurable) subset    $ \Omega_3\subset \Omega$ with $P(\Omega_3)=1$
   such that for any $\omega\in \Omega_3$, any $k$ and any $a_j \in M_k$ 
   we have
   $$  \max\{\| (\sum a_j^*a_j)^{1/2} \|,\| (\sum a_ja_j^*)^{1/2}\| \}   \le  \liminf\nolimits_{N\to \infty} \left\|\sum_1^n Y_j^{(N)} \otimes a_j\right\|. $$
   Indeed, we can easily reduce this to the countable set of all $a_j$'s with entries in (say)   $\Q+i\Q$.\\
   A fortiori, for any $\omega\in \Omega_3$ any $k$, any $a_j \in M_k$ and any infinite subset $\a\subset \NN$  we have
   $$  \max\{\| (\sum a_j^*a_j)^{1/2} \|,\| (\sum a_ja_j^*)^{1/2}\| \}   \le \|\sum a_j \otimes u_j(\a)  (\omega)\|.$$
   Thus the mapping $v:\ E(\a)  (\omega)\to R_n$ (resp. $w:\ E(\a)  (\omega)\to C_n$)
   defined by $v(u_j(\a)  (\omega))= e_{1j}$ (resp. $w(u_j(\a)  (\omega))= e_{j1}$) satisfies
       \begin{equation}\label{es5}  \|v\|_{cb}\le 1 \quad{\rm (resp.\  }\|w\|_{cb}\le 1).
        \end{equation}

   Let $\Omega''\subset \Omega$ be the set of all $\omega$'s such that
                    $$\limsup_{m\to \infty} \|(  Y_j^{(N(m))}(\omega))\|_{RC} \le 2\sqrt{n}.$$
                    Clearly $\P( \Omega'')=1$  (Indeed 
                    this follows a fortiori from Lemma \ref{lemcon}).\\
                    Let  $\Omega'''\subset \Omega$ be the set of all $\omega$'s such that for any matrix 
                    $a\in M_n$ we have
                    $$|{\rm tr} (a)|\le \liminf_{m\to\infty} |\sum\nolimits_{ij} a_{ij} N(m)^{-1} {\rm tr} (Y_i^{(N(m))}{Y_j^{(N(m))}}^*)|.$$ The convergence in moments (to a circular family)
                    of $Y_j^{(N(m))}$ when $m\to \infty$ ensures that 
                    this event has full probability for any fixed $a$, but again a density argument
                    (in $M_n$) ensures that $\P( \Omega'''\  )=1$.
                    
                    For convenience we denote
        $$\Delta_j= e_{1j}\oplus e_{j1}\in R_n \oplus C_n.$$
    We now choose $\omega$ in the set $\Omega_2\cap \Omega_3\cap \Omega'' \cap\Omega''' $ which occurs with full probability and we set $$x_j(m)=Y_j^{(N(m))}(\omega)\in M_{N(m)}.$$
                  By the choice of $\omega$ we know that for some $m_0>1$
                  we have for any $m\ge m_0-1$ and any $(y_j)\in M_{N(m)}^n$
                   \begin{equation}\label{eq22} \|y\|_{RC}\le \sup_{m'>m} \|\sum x_j(m') \otimes y_j\|  \le (2+\vp)\|y\|_{RC} =  (2+\vp)\|\sum \Delta_j \otimes y_j\|  \end{equation}
        so that
                    \begin{equation}\label{eq222}  
                     \sup_{1\le m'\le m} \|\sum x_j(m') \otimes y_j\|\le  \|\sum x_j \otimes y_j\|  \le \max\{
                     \sup_{1\le m'\le m} \|\sum x_j(m') \otimes y_j\|, (2+\vp)\|\sum \Delta_j \otimes y_j\| \}
                    \end{equation}
For convenience we may as well assume $m_0$ large enough so that $N(m_0-1)\ge 2n$. Then \eqref{eq22} implies that for
any  $(y_j)\in (R_n\oplus C_n)^n$
  \begin{equation}\label{eq223}\sup_{m'\ge m_0} \|\sum x_j(m') \otimes y_j\|  \le (2+\vp)\|\sum \Delta_j \otimes y_j\|.
\end{equation}
 For any subset $\a \subset \NN$, we will denote
         \begin{equation}\label{eq19} 
        x_j(\a)=\oplus_{m \in \a }  x_j(m)\quad{\rm and}\quad E_\a={\rm span}\{x_1(\a),\cdots,x_n(\a)\}. \end{equation}

           \def\b{\beta}         The following was used in \cite{JP} for diagonal matrices, the general case   was observed in \cite{OR}.
         
            \begin{lem}\label{or}     With the above notation, let $\a,\b\subset  [m_0,\infty)$ be infinite  such that
          $\a \setminus\b$ is infinite.  
          For any    $[a_{ij}]\in M_n$ let $T:\ E_\b\to E_\a$ be defined by $T(x_j(\b))=\sum_i  a_{ij} x_i(\a)$. Then we have
          $$ (2+\vp)^{-2} (\sum |  a_{ij}|^2)^{1/2}\le  \|T\|_{CB(E_\b, E_\a)}\le (2+\vp) (\sum |  a_{ij}|^2)^{1/2}.$$ 
        \end{lem}
        \begin{proof}  
        Since   $\omega\in \Omega_3$,   \eqref{es5} ensures a fortiori that
           we have  a completely contractive natural  map $E_\b \to R_n$
           (and also $E_\b \to C_n$) for any infinite $\b$. Therefore $\|T\|_{CB(E_\b, E_\a)}\le \|T\|_{CB(R_n,E_\a )}=\|\sum a_{ji}  e_{i1}  \otimes u_j(\a)(\omega)\|$
           and   by \eqref{eq223}
                      $$  \|\sum a_{ji}  e_{i1}  \otimes u_j(\a)(\omega)\|\le  (2+\vp) (\sum |  a_{ij}|^2)^{1/2}.$$
            Therefore $\|T\|_{CB(E_\b, E_\a)}\le (2+\vp) (\sum |  a_{ij}|^2)^{1/2}.$

        To prove the converse, we will choose $m'_0\ge m_0$ arbitrarily large in $\a \setminus \b$. 
        Then  \eqref{eq22}  and \eqref{eq223} show that for any scalar matrix $[b_{jk}]$
 $$\|\sum b_{jk}  x_j(\b\cap (m'_0,\infty))\otimes \ov{ x_k(m'_0)}\|\le (2+\vp) \|\sum b_{jk}  \Delta_j\otimes  \ov{ x_k(m'_0)}\|\le (2+\vp)^2 \|\sum b_{jk}  \Delta_j\otimes \ov{\Delta_k} \|,$$  
 and also since $m'_0\not\in \b$ again by \eqref{eq22}  and \eqref{eq223}
 $$\|\sum b_{jk}  x_j(\b\cap [m_0,m'_0])\otimes  \ov{ x_k(m'_0)}\|\le (2+\vp) \|\sum b_{jk}  x_j(\b\cap [m_0,m'_0])\otimes \ov{\Delta_k}\|$$
 $$\le (2+\vp)^2 \|\sum b_{jk}  \Delta_j\otimes\ov{\Delta_k}\|.$$ 
     Recollecting the two preceding estimates, we find
 $$  \|\sum b_{jk}  x_j(\b)\otimes  \ov{ x_k(m'_0)}\|  \le  (2+\vp)^2 \|\sum b_{jk}  \Delta_j\otimes \ov{\Delta_k}\|=(2+\vp)^2 (\sum\nolimits_{jk} |b_{jk}|^2 )^{1/2} .$$
 Thus we have
 $$\|\sum\nolimits_{ijk}  b_{jk} a_{ij} x_i(\a)\otimes  \ov{ x_k(m'_0)}\| =\|\sum b_{jk}  T(x_j(\b))\otimes  \ov{ x_k(m'_0)}\|\le \|T\|_{cb}  (2+\vp)^2 (\sum\nolimits_{jk} |b_{jk}|^2 )^{1/2}.$$
 Since $m'_0\in \a$,  a fortiori $$\|\sum_{ijk} b_{jk} a_{ij} x_i(m'_0)\otimes  \ov{ x_k(m'_0)}\| \le \|T\|_{cb}  (2+\vp)^2 (\sum\nolimits_{jk} |b_{jk}|^2 )^{1/2}.$$
 But now,  since $\omega \in \Omega'''$, for any $\delta>0$, when $m'_0$ is chosen large enough, we
 have 
 $$|\sum\nolimits_{ij} a_{ij} b_{ji}|\le \|\sum\nolimits_{ijk} b_{jk} a_{ij} x_i(m'_0)\otimes  \ov{ x_k(m'_0)}\| +\delta$$
 and hence choosing simply $b_{ij}=\bar a_{ji}$ (and letting $\delta \to 0$)  we conclude
 $$\sum\nolimits_{ij} |a_{ij} |^2 \le \|T\|_{cb}(2+\vp)^2(\sum\nolimits_{ij} |a_{ij} |^2)^{1/2}$$
 and the announced result follows after division.
    \end{proof}            
           
            \begin{lem}\label{ors+}     Fix $m\ge m_0-1$. Assume that $\a,\b$ are infinite subsets such that
            $\a\cap [0,m]= \b\cap [0,m]$. Then the mapping $T:\ E_\b\to E_\a$
            (induced by the identity of $\CC^n$) is such that
            $$\forall k\le N(m)\quad \|T_{k}:\ M_{k} (E_\b) \to M_{k} (E_\a) \|\le 2+\vp   $$
        \end{lem}
        \begin{proof} For all $y\in M_k^n$ we have
    $$\| \sum x_j(\a\cap [0,m] ) \otimes y_j \|=\| \sum x_j(\b\cap [0,m] ) \otimes y_j \|
    \le \| \sum x_j(\b ) \otimes y_j \|.$$
    But also by 
     \eqref{eq22} since $k\le N(m)$ and $\omega\in \Omega_3$
    $$ \| \sum x_j(\a\cap [m+1,\infty) ) \otimes y_j \|\le (2+\vp) \|y\|_{RC}\le (2+\vp) \| \sum x_j(\b ) \otimes y_j \|.$$
    Therefore $\|T_k\|\le 2+\vp.$
   \end{proof}
        
      The method of the paper \cite{JP} as presented in \cite[Th. 21.13, p. 343]{P4}  shows that
        there is a continuous subcollection in the family
         $\{E_\a\mid \a\subset \NN, |\a|=\infty\}$ formed of spaces $E$'s such that $ex(E)\ge  \sqrt{n}/(2 +\vp)^3$.
         But actually we can make this slightly more precise:

           \begin{thm}\label{blu}  
           For any infinite subset $\a\subset [m_0,\infty)$ 
          we have $$ex(E_\a)\ge  \sqrt{n}/(2 +\vp)^3.$$
        \end{thm}
        \begin{proof}    
        Fix $\a$ as in the statement. For any $m\ge m_0-1$
        let
        $${\cl C}_m=\{    \b\subset [m_0,\infty) \mid |\b|=\infty,   \ \a\cap [0,m]= \b\cap [0,m],\ 
       | \a\setminus\b|=\infty
        \}.$$ Clearly  this set is non empty.
 For each $m \ge m_0-1$, pick $\beta(m) \in  {\cl C}_m$. 
 Let $T(m): E_{\beta(m)}\to E_\a$ denote the natural (identity) map.
 By   Lemma \ref{or} we have
 $\| T(m)\|_{cb}   \ge \sqrt{n}/(2+\vp)^2$. However, by Lemma \ref{ors+}, for any $k\le N(m)$
 we have 
 $\| T(m)_{k}\|   \le  (2+\vp)$.
 Assume now that $E_\a$ is $C$-exact, so that for some finite $k$ there is
 $F\subset M_k$ with $d_{cb} (E_\a,F)\le C$. Choosing $m$ large enough
 we may ensure that $k\le N(m)$. By the Smith Lemma \ref{rs},
 this implies that $\|T(m)\|_{cb}\le C \| T(m)_{k}\| $.
 Thus we obtain $\sqrt{n}/(2+\vp)^2\le C (2+\vp)$, and hence $C\ge  \sqrt{n}/(2 +\vp)^3$.
  \end{proof}
  Recapitulating, we can now conclude  
    \begin{thm}\label{blue} For any $n\ge 1$ and $\vp>0$, there is a continuum
    of $n$-dimensional, subexponential
     (with  constant $2+\vp$)   operator spaces with mutual $cb$-distance $\ge   {n}/(2 +\vp)^4$
    and with exactness constant $\ge \sqrt{n}/(2 +\vp)^3$.
                  \end{thm}
        Note that the preceding lower bounds are not significant   for small values of $n$.

           \begin{rem}\label{bblu}  
          Note that the preceding result can also be obtained using the main idea of \cite{P5}.
          In fact that idea proves more generally that for $\a$ infinite with infinitely many gaps
          the space $E_\a$ has a large $d_f$ constant of embedding into $C^*(\F_2)$ in the sense
          of \cite[p. 345]{P4}.         
        \end{rem}

       \begin{cor}\label{bbla}     For any $\vp>0$ there is a separable (infinite dimensional) non-exact
       operator space which is $(2+\vp)$-subexponential,
        more precisely  such that 
       any finite dimensional subspace $E\subset X$ satifies  $K_E(N,2+\vp)\in O(N^2)$. 
        
        \end{cor}
        \begin{proof}   Let $E(n)$ be any one of the  $n$-dimensional spaces appearing  in Theorem \ref{blue}.
        Let $X$ be the direct sum in the $c_0$-sense of $\{E(n)\}$, so that the elements of
        $X$ are sequences $x=x(n)$ tending to zero in norm and $X\subset \oplus_n E(n)$.
         We claim that any finite dimensional 
        subspace   $E\subset X$ is subexponential with constant $2+\vp$. By perturbation, it suffices to show this
        for $E= \oplus_{[0\le n\le q]} E(n)$, for any integer  $q$. But now
        an easy verification shows that for such an $E$ we have
        $$K_E(N,C) \le \sum_{0\le n\le q} K_{E(n)}(N,C)  $$
        and since each $E(n)$ is subexponential with constant $2+\vp$, we conclude
        that $E$ also is.
         \end{proof}
We will now replace our Gaussian random matrices by random unitary ones.
Although the picture is less precise, we obtain some information using a rather soft comparison principle.
     Let $U^{(m)}_j$ be an i.i.d. sequence of random unitary matrices uniformly distributed over the unitary group $U(m)$.
     There is a   known domination of $\{U^{(m)}_j\}$ by $\{Y^{(m)}_j\}$   (up to a universal constant) that will allow us to obtain results
      similar to what precedes for $\{U^{(m)}_j\}$ in place of $\{Y^{(m)}_j\}$.  
      
      Fix an integer $n$. Let $D_Y(m)(\omega) $ (resp. $D_U(m)(\omega) $) be  
      defined by
      $$      D_Y(m)(\omega) =\sup_{k\le (m/c_\vp n)^{1/2}} \{ \|\sum a_j \otimes  Y^{(m)}_j \|\mid a=(a_j)\in M_k^n \ \|a\|_{RC}\le 1  \}$$
      and
      $$      D_U(m)(\omega) =\sup_{k\le (m/c_\vp n)^{1/2}} \{ \|\sum a_j \otimes  U^{(m)}_j \|\mid a=(a_j)\in M_k^n \ \|a\|_{RC}\le 1\},$$
      where $c_\vp$ is as in Lemma \ref{blabla}.
      As we saw in Lemma \ref{blabla} ,  we have for a.a. $\omega$ (in fact for all $\omega\in \Omega_2$)
       $\limsup_{m\to \infty } D_Y(m)(\omega)\le 2+\vp $ and a fortiori 
        $\sup_m D_Y(m)(\omega)<\infty.$
     A close look at the proof of
      Lemma \ref{lemcon} (see also \eqref{eq30})  shows that we have
  actually  for  any $p\ge 1$
$$\E \sup_m D_Y(m)^p  <\infty,$$
and hence by dominated convergence
$$\lim_{m'\to \infty} (\E \sup_{m\ge m'} D_Y(m)^p)^{1/p} \le 2+\vp.$$
The domination principle of $\{U^{(m)}_j\}$ by $\{Y^{(m)}_j\}$ described in \cite[p. 84]{MP} implies
that  for  any $p\ge 1$ we have for all integers $q$   $$(\E \sup_{m\ge q} D_U(m)^p)^{1/p}\le \chi (\E \sup_{m\ge q} D_Y(m)^p)^{1/p}$$
where $\chi$ is a numerical constant. Thus we obtain for  any $p\ge 1$
$$\lim_{q\to \infty} (\E \sup_{m\ge q} D_U(m)^p)^{1/p}  \le \chi(2+\vp),$$
and a fortiori for a.a. $\omega$ 
      $$\limsup_{m\to \infty } D_U(m)(\omega)\le \chi(2+\vp) .$$
      Using this it is immediate that Lemma \ref{ors} remains valid
   for $\{U^{(m)}_j\}$ in place of $\{Y^{(m)}_j\}$ provided we replace $2+\vp$ by $ \chi(2+\vp) $.
     Note that in the unitary setting the  free circular family $(c_j)$ should be replaced by a free
     (actually $*$-free)  sequence of   Haar  unitaries.
     \def\l{\lambda}

     Let us denote by   $E_U(\a_2)(\omega)$
       the space appearing in Lemma \ref{ors} with  $\{U^{(m)}_j\}$ in place of $\{Y^{(m)}_j\}$. 
       Let $K'_N$ be the largest integer such that $N^{-1/2} K'_N< (c_\vp n)^{-1/2}$.
Let $\lambda_N(\omega)$ be the smallest number $\l$ such that
    $N^{-1/2} K_{E_U(\a_2)(\omega)}(N,\l)< (c_\vp n)^{-1/2}$
    or equivalently such that $K_{E_U(\a_2)(\omega)}(N,\l)\le    K'_N$. In other words,
       $$\lambda_N(\omega)=\inf\{ d_N(E_U(\a_2)(\omega), \hat E)\mid {\hat E\subset M_{K'_N} }\}.$$
       Note that using the separability of the 
       various underlying metric spaces involved, one can rewrite the preceding
       infimum (as well as $d_N$)  as a countable infimum,        yielding the measurability of   $\lambda_N$.       We then let
       $$\lambda(\omega)=\liminf\nolimits_{N\to \infty} \lambda_N(\omega).$$ 
       \begin{lem}\label{ors2}    
       With the preceding notation,  for a.a. $\omega$  $$\sqrt n \le \chi(2+\vp)   \l(\omega) ,$$
       and hence          whenever $C< \sqrt n  ( \chi(2+\vp))^{-1}$  we have
       $$\liminf_{N\to \infty} N^{-1/2} K_{E_U(\a_2)(\omega)}(N,C)\ge (c_\vp n)^{-1/2}.$$
       \end{lem}
        \begin{proof} 
       Recall that if we set
        $$\Omega^U_2=\{ \omega\in \Omega\mid \limsup_{m\to \infty } D_U(m)(\omega)\le \chi(2+\vp)\},$$
          we have $\P(\Omega^U_2)=1$.  Fix $\omega\in \Omega^U_2$. Let $N(m)$ be as defined in \eqref{eq25}.
        Let $v_j(m)=U_j^{(N(m))}$ and $v_j =\oplus_{m\in \a_2} v_j(m)$,  so that  $E_U(\a_2)={\rm span}\{v_1,\dots,v_n\}$.
        For simplicity let $E=E(\omega)=E_U(\a_2)$.
         Fix   $m$.  Let $K=K_{E_U(\a_2)(\omega)}(m,\l_m)$. 
         By definition of $\l_m$ we have $c_\vp n K^2<m$.
       Moreover,  there is $\hat E\subset M_{K} $ such that $ d_m(E,\hat E) \le \l_m(\omega)  $. So we can find
        a linear isomorphism $w:\ E\to \hat E$ such that $\|w_m\|\le 1$ and $\|(w^{-1})_m\|\le \l_m$.
        Let $v'_j=w(v_j)$. Note that  $\|v'_j\|\le 1$ for any $j$ (since $v_j$ is unitary).
        We have
        $$n=\|\sum v_j(m)\otimes \overline { v_j(m)}\|\le \|\sum v_j\otimes \overline { v_j(m)}\|\le \l_m(\omega)\|\sum v'_j\otimes \overline { v_j(m)}\| ,$$
        and hence        since    
         $K\le [   \sqrt{m/c_\vp n}  ]   $  and    $v'_j\in   M_{K}  $ we have
       $$n \le  \|\sum v'_j\otimes \overline { v_j(m)}\| \le \l_m D_U([   \sqrt{m/c_\vp n}  ]) \|(v'_j)\|_{RC}\le \l_m  D_U( [\sqrt{m/c_\vp n}]) \sqrt n  , $$
       therefore   since 
          $\omega\in \Omega_2$
        $$n \le  \liminf_m \l_m \limsup_m D_U([   \sqrt{m/c_\vp n}  ]) \sqrt n \le\l \chi(2+\vp) \sqrt n,$$ 
        and hence 
         $ \sqrt n\le\l \chi(2+\vp).$ 
                        \end{proof}
       Thus as a recapitulation we may state:
     \begin{thm}\label{cm}   The space  $E(\omega)=E_U(\a_2)(\omega)$ is a.s. 
     $ \chi(2+\vp) $-subexponential   but   not $C$-exact,
  whenever $C< \sqrt n   ( \chi(2+\vp))^{-1}$.   Moreover if $n$ is  chosen sufficiently large    
         (so that, say $ \sqrt n   ( \chi(2+\vp))^{-1}>3$) and $\vp<1$,  then a.s.
         $$  \limsup_N  K_{E(\omega)}(N,3)/N^2<\infty  \quad{\rm and}\quad  \liminf_N  K_{E(\omega)}(N,3)/N^{1/2} >0.$$
                         
    \end{thm}
    
           \begin{rem}\label{r22}  A similar result can be proved with large probability
           for the space $E_Y(\a_2)(\omega)$ but the fact that   $\|v_j\|=1$ makes 
           the preceding result easier to check, so we do not give more details.
       \end{rem}
      \section{Beyond exact or subexponential}\label{s8}
      The preceding results highlight the fact  that 
       \begin{equation}\label{fin2}||| (a_1,\cdots,a_n)|||= \limsup\nolimits_{N\to \infty} \left\|\sum_1^n  Y_j^{(N)}\otimes a_j \right\|,\end{equation}
       is equal to $\|\sum c_j\otimes a_j\|$ when $E={\rm span}[a_j]$ is exact or subexponential with constant 1.
       It is natural to wonder what happens when $E$ is arbitrary.
  In this section we make a preliminary study in this direction. In particular,
  we will see that it is rather easy to estimate \eqref{fin2} when $E=\ell_1^n$ (with maximal o.s.s.)
  or when $E=OH_n$. We also propose a general rough estimate based on the numbers
  $K_E(N,C)$.
   \begin{thm}\label{l22}  Let $[a_{ij}]$ be any $n\times n$-matrix with complex entries.
  \item{(i)}
  Let $U_i$ denote the free unitary generators in $
   C^*(\F_n)$.  Let $a_j=\sum_i a_{ij} U_i$. Then
   $$2^{-1} \sum\nolimits_i (\sum\nolimits_j |  a_{ij} |^2)^{1/2}\le  ||| (a_1,\cdots,a_n)|||\le 2  \sum\nolimits_i (\sum\nolimits_j |  a_{ij} |^2)^{1/2}.$$
     \item{(ii)} Let $T_i$ be any orthonormal basis in $OH_n$.
     Let $a_j=\sum_i a_{ij} T_i$. Then
     $$    (\sum\nolimits_{ij} |  a_{ij} |^2)^{1/2}\le  ||| (a_1,\cdots,a_n)|||\le 2   (\sum\nolimits_{ij} |  a_{ij} |^2)^{1/2}.$$
 \end{thm}
       \begin{proof} (i) By the triangle inequality we have
      $ ||| (a_1,\cdots,a_n)|||\le \sum_i \limsup\nolimits_{N\to \infty}    \|\sum_j a_{ij}  Y_j^{(N)} \| $.
      Let $W^{(N)}_i=\sum_j a_{ij} { Y_j^{(N)}}$. 
     Since $W^{(N)}_i$ has the same distribution as  $ (\sum\nolimits_{j} |  a_{ij} |^2)^{1/2} Y^{(N)}$
     and since  $ \limsup\nolimits_{N\to \infty}  \|Y^{(N)}\|= 2$, we obtain the upperbound in (i).
        Let  $[b_{ij}]$ be such that $\sup_i (\sum\nolimits_j |  b_{ij} |^2)^{1/2}\le1$. 
         To prove the lower bound, it suffices to show that
         $2^{-1}|\sum\nolimits_{ij} \ov{b_{ij}} a_{ij}| \le  ||| (a_1,\cdots,a_n)|||$.
         Let $Z^{(N)}_i=\sum_k b_{ik} { Y_k^{(N)}}$. Note that since $Z^{(N)}_i$ has the same distribution
         as $(\sum\nolimits_j |  b_{ij} |^2)^{1/2} Y^{(N)} $ we have
       $ \limsup\nolimits_{N\to \infty} \|  Z^{(N)}_i \|\le 2$.
         We have clearly  
         $$|\sum_{ij} a_{ij}  \tau_N ( Y_j^{(N)} {Z^{(N)}_i}^*) |\le \|\sum_{ij} a_{ij}   Y_j^{(N)} \otimes  \ov{  Z^{(N)}_i}\|\le \sup_i \|Z^{(N)}_i\| \|\sum_{ij} a_{ij}   Y_j^{(N)} \otimes   U_i\| $$
        and hence taking the limsup of both sides when $N\to \infty$
        we find by \eqref{voi}
        $$| \sum\nolimits_{ij} \ov{b_{ij}} a_{ij}|\le 2  ||| (a_1,\cdots,a_n)|||  $$
        which proves (i).\\
        (ii) Let $W^{(N)}_i$ be as above.
        Note that $\|   \sum_{ij} a_{ij}   Y_j^{(N)} \otimes T_i\| =\|\sum W^{(N)}_i \otimes \ov{ W^{(N)}_i}\|^{1/2}$.
       Therefore
       $|\sum _{i}  \tau_N ( W_i^{(N)} {W^{(N)}_i}^*) |^{1/2}\le \|   \sum_{j}  Y_j^{(N)} \otimes a_{j}  \|$
       and hence taking the limsup of both sides when $N\to \infty$
       we   find the lower bound in (ii). For the converse, note that
       $\|\sum W^{(N)}_i \otimes \ov{ W^{(N)}_i}\|^{1/2}\le (\sum\| W^{(N)}_i \|^2)^{1/2}$
       and hence taking  the limsup of both sides when $N\to \infty$ we find the upperbound in (ii).
     \end{proof}
        Recall first that by \eqref{s2}  $  \|\sum_j \Delta_j \otimes a_j\|$ is equivalent,  up to a factor 2, to $  \|\sum_j c_j \otimes a_j\|$. Thus,
       the preceding should be compared with the
        following:        
        \begin{equation}\label{fin3}\|\sum_j \Delta_j \otimes\sum_i a_{ij} U_i \|=\pi_2([a_{ij} ]:\ \ell_\infty^n \to \ell_2^n) \quad{\rm and}\quad\|\sum_j \Delta_j  \otimes\sum_i a_{ij} T_i \|=\|[a_{ij}]\|_{4},\end{equation}
        where $\pi_2$ denotes the 2-summing norm and $\|.\|_4$ the norm in the Schatten class of index 4. Note that, by the little Grothendieck theorem (see e.g. \cite{Pgt}),
        there is an absolute constant, namely $\gamma(1)^{-1}$, such that 
        $\|[a_{ij} ]:\ \ell_\infty^n \to \ell_2^n\|\le \pi_2([a_{ij} ]:\ \ell_\infty^n \to \ell_2^n)\le \gamma(1)^{-1}\|[a_{ij} ]:\ \ell_\infty^n \to \ell_2^n\|$. We refer the reader
        to   e.g. \cite[\S 5 and Th. 13.10]{Pgt} for this fact and for the first formula
        in \eqref{fin3}
        We refer to \cite[p. 38]{P3} for the second one.
         \begin{rem}\label{l23} 
        We end this paper with a majorization that can be applied
        to estimate  $||| (a_1,\cdots,a_n)|||$ for arbitrary $a_j\in B(H)$. This
          illustrates the possible applications of bounds such as \eqref{eq30}.
      Let $E$ denote the linear span of an arbitrary $n$-tuple $(a_1,\cdots,a_n)$ in $B(H)$.
       Let $i_E:\ E \to B(H)$ denote the inclusion of $E$ into $B(H)$. Fix $N\ge 1$. Assume that
        we have sequences of maps $v(m):\ E\to M_{K(m)}$, $w(m):\ M_{K(m)}\to B(H)$ 
        such that $\sum_{m} \|w(m) \|  \|v(m)\|<\infty$ and  $i_E=\sum_{m} w(m) v(m)$.
       Then the following is an immediate consequence of   \eqref{eq30}.
           \begin{equation}\label{eq303}\E\left \|   \sum\nolimits_1^n Y_j^{(N)} \otimes a_j\right\|\le (1+\vp)
      \sum_{m}  \left(2 +\gamma_\vp(\frac{\log({K(m)})+1}{ N})^{1/2} \right) \|(w(m)_N\| (v(m)(a_j))\|_{RC}.\end{equation} 
      \end{rem}

        \bigskip
  
        \n\textbf{Acknowledgment.}  I thank Kate Juschenko for useful conversations at an early stage of this investigation. I am very grateful to Mikael de la Salle for numerous remarks
        and suggestions  that led to many improvements.
        Thanks are due also to
         Yanqi Qiu, Alexandre Nou and  the anonymous referee for several corrections.


\begin{thebibliography}{100}
               \bibitem{BO} N.P. Brown and N. Ozawa, \emph{$C^*$-algebras and finite-dimensional approximations}, Graduate Studies in Mathematics, 88, American Mathematical Society, Providence, RI, 2008. 
             
             \bibitem{Buch} A. Buchholz,   Operator Khintchine inequality in non-commutative probability.  \emph{Math. Ann.}  {\bf 319}  (2001),    1--16.


                 \bibitem{CM}   B. Collins and C. Male, The strong asymptotic freeness of Haar and deterministic
matrices, To appear.
       \bibitem{ER} 
E.G. Effros and Z.J. Ruan, \emph{Operator Spaces},  The Clarendon Press, Oxford University Press, New York, 2000, xvi+363 pp.
      
    
        \bibitem{FLM} T. Figiel, J. Lindenstrauss and V.  Milman. The dimension of almost spherical sections of convex bodies. Acta Math. 139 (1977),   53-94.
       \bibitem{G}  S. Geman,  
             A Limit Theorem for the Norm of Random Matrices, Ann. Prob.,   8  (1980), 252-261.
 
 \bibitem{GM} M. Gromov and V. Milman, A Topological Application of the Isoperimetric Inequality,
 American Journal of Mathematics,   105  (1983), 843-854.
 
 
      \bibitem{HT2} 
U. Haagerup and S. Thorbj{\o}rnsen, Random matrices and $K$-theory for exact $C^*$-algebras, \emph{Doc. Math.} {\bf 4} (1999), 341--450 (electronic).
\bibitem{HT3} 
U. Haagerup and S. Thorbj{\o}rnsen, A new application of random matrices: {${\rm Ext}(C^*_{\rm
  red}(\mathbb F_2))$} is not a group.
\emph{Ann. of Math. } 162 (2005), 711--775.

\bibitem{HT4} 
U. Haagerup, H. Schultz and S. Thorbj{\o}rnsen, 
 A random matrix approach to the lack of projections in $C^*_{\rm red}(\F_2)$. Adv. Math. 204 (2006), no. 1, 1-83. 
 

      \bibitem{HM2} 
U. Haagerup  and M.   Musat, The Effros--Ruan conjecture for bilinear forms on $C^*$-algebras,  \emph{Invent. Math.}  {\bf 174} (2008),   139--163.



            \bibitem{JP} 
M. Junge and G. Pisier,  Bilinear forms on exact operator spaces and $B(H)\otimes B(H)$, \emph{Geom. Funct. Anal.} {\bf 5} (1995), no. 2, 329--363.  

     \bibitem{Led} 
M. Ledoux,  \emph{The Concentration of Measure Phenomenon},  American Mathematical Society, Providence, RI, 2001. 

        \bibitem{M} C. Male,  The norm of polynomials in large random and deterministic matrices
 with an appendix by  Dimitri Shlyakhtenko, arXiv:1004.4155v5
  

        \bibitem{MP} M.B. Marcus and G.Pisier, 
         {\it Random Fourier series with Applications to
Harmonic Analysis.}  Annals
of Math. Studies n$^\circ$101, Princeton University Press.
(1981).
        
       
 \bibitem{OR} T. Oikhberg and \'E. Ricard, Operator spaces with few completely bounded maps. Math.
Ann., 328 (2004) 229-259.
 
         
          \bibitem{Pm} 
G. Pisier, Remarques sur un r\'esultat non publi\'e de B. Maurey. Seminar on Functional Analysis, 1980-1981, Exp. No. V, 13 pp., \'Ecole Polytech., Palaiseau, 1981. [Available at http://www.numdam.org/numdam-bin/feuilleter?j=SAF]

     
  \bibitem{P02} 
G. Pisier,   Probabilistic methods in the geometry of Banach spaces, \emph{Probability and Analysis (Varenna, 1985)}, 167--241, \emph{Lecture Notes in Math.}, 1206, Springer, Berlin, 1986.

 \bibitem{P-v} G. Pisier, {\it The volume of Convex Bodies and Banach
Space Geometry}.   Cambridge University Press, 1989.

\bibitem{P3}
  G. Pisier,  The operator Hilbert space ${\rm OH}$, complex interpolation and tensor norms, \emph{Mem. Amer. Math. Soc.} {\bf 122} (1996), no. 585. 


  \bibitem{P4}  
G. Pisier,  \emph{Introduction to operator space theory},  Cambridge University Press, Cambridge, 2003. 

               \bibitem{P5}  
G. Pisier, Remarks on $B(H)\otimes B(H)$.
Proc. Indian Acad. Sci. 116 (2006),  423--428.

 \bibitem{Pgt} G. Pisier, Grothendieck's theorem, past and present.
 Bull. Amer. Math. Soc. 49 (2012), 237-323.
  \bibitem{Pnew} G. Pisier,  Martingale inequalities and Operator space structures on $L_p$,
  Preprint, 2012.
\bibitem{Pq2} G.  Pisier, Quantum expanders and geometry of operator spaces, To appear.
\bibitem{PS}
G.  Pisier and D. Shlyakhtenko,  Grothendieck's theorem for operator spaces, \emph{Invent. Math.} {\bf 150} (2002), no. 1, 185--217.


 \bibitem{RV} O. Regev and T. Vidick, A simple proof of Grothendieck's theorem for completely bounded norms, Preprint 2012, to appear in J. Op. Theory. 
  \bibitem{S}  
H. Schultz, Non-commutative polynomials of independent Gaussian random matrices, The real
and symplectic cases, Probab. Theory Rel. Fields 131 (2005) 261Ð309.

             \bibitem{VDN} D. Voiculescu, K. Dykema and A.  Nica, \emph{Free random variables},  American Mathematical Society, Providence, RI, 1992.
 \end{thebibliography}
  \end{document}